\documentclass[11pt,a4paper]{article}
\usepackage[utf8]{inputenc}
\usepackage[english]{babel}
\usepackage[T1]{fontenc}
\usepackage{amsmath, amsfonts, amssymb, amsthm}
\usepackage{lmodern}
\usepackage{fullpage}
\usepackage{microtype}
\usepackage{enumerate}
\usepackage{mathtools}
\usepackage{cases}

\usepackage{tikz}
\usetikzlibrary{matrix, arrows}
\usepackage{array}
\newcolumntype{M}[1]{>{\centering\arraybackslash}m{#1}}

\usepackage[colorlinks=true, linkcolor=blue]{hyperref}

\numberwithin{equation}{section}

\newtheorem{theoreme}[equation]{Theorem}
\newtheorem{proposition}[equation]{Proposition}
\newtheorem{lemme}[equation]{Lemma}
\newtheorem{corollaire}[equation]{Corollary}

\theoremstyle{definition}
\newtheorem{definition}[equation]{Definition}

\theoremstyle{remark}
\newtheorem{remarque}[equation]{Remark}

\renewcommand*{\H}{\mathcal{H}}
\newcommand*{\sH}{\mathcal{H}} 
\newcommand*{\qH}{\mathcal{R}}
\newcommand*{\sqH}{\mathcal{R}}
\newcommand*{\V}{V}
\newcommand*{\tuple}[1]{\boldsymbol{#1}}
\newcommand*{\comp}{\models}
\renewcommand*{\varLambda}{\mathsf{\Lambda}}
\newcommand*{\vertex}[1]{#1}

\DeclareMathOperator{\nrelbar}{\,\not\!\!\!\!\;\text{---}\,}

\author{Salim \textsc{Rostam}}
\title{Cyclotomic quiver Hecke algebras and Hecke algebra of $G(r,p,n)$}
\date{}
\author{Salim \textsc{Rostam}\thanks{Laboratoire de Mathématiques de Versailles, UVSQ, CNRS, Université Paris-Saclay, 78035 Versailles, France.}}

\begin{document}
\maketitle

\abstract{Given a quiver automorphism with nice properties, we give a presentation of the fixed point subalgebra of the associated cyclotomic quiver Hecke algebra. Generalising an isomorphism of Brundan and Kleshchev between the cyclotomic Hecke algebra of type $G(r, 1, n)$ and the cyclotomic quiver Hecke algebra of type A, we apply the previous result to find  a  presentation of the cyclotomic Hecke algebra of type $G(r,p,n)$ which looks very similar to the one of a cyclotomic quiver Hecke algebra. In addition, we give an explicit isomorphism which realises a well-known Morita equivalence between Ariki--Koike algebras.}

\section*{Introduction}
\addcontentsline{toc}{section}{Introduction}

Generalising real reflection groups, also known as finite Coxeter groups, complex reflection groups are finite groups generated by complex reflections, that is, by endomorphisms of $\mathbb{C}^n$ which fix a hyperplane. As for Coxeter groups, there is a classification of irreducible complex reflection groups (\cite{ShTo}). This classification is given by an infinite series $\{G(r,p,n)\}$ where $r,p,n$ are positive integers with $r = dp$ for $d \in \mathbb{N}^*$, together with 34 exceptional groups. More precisely, the group $G(r,p,n)$ can be seen as the group consisting of all  $n \times n$ monomial matrices such that each non-zero entry is a complex $r$th root of unity, and the product of all non-zero entries is a $d$th root of unity. If $\xi \in \mathbb{C}^\times$ is a primitive $r$th root of unity, the latter group is generated by the elements:
\begin{gather*}
s \coloneqq \xi^p E_{1,1} + \sum_{k = 2}^n E_{k, k}, \qquad \widetilde{t}_1 \coloneqq \xi E_{1, 2} + \xi^{-1} E_{2, 1} + \sum_{k = 3}^n E_{k, k},
\\
t_a \coloneqq E_{a, a+1} + E_{a+1, a} + \sum_{\substack{1 \leq k \leq n \\ k \neq a, a+1}} E_{k, k}, \qquad \text{for all } a \in \{1, \dots, n-1\},
\end{gather*}
where $E_{k, \ell}$ is the elementary $n\times n$ matrix with $1$ as the $(k, \ell)$-entry and $0$ everywhere else.

Inspired by work in the context of finite Chevalley groups, an algebra was associated with each real reflection group $W$: the \emph{Iwahori--Hecke algebra} $\H(W)$. This deformation of the group algebra of $W$ was the starting point of many connections with other objects and theories, for instance, the theory of quantum groups or knot theory. Aiming at generalising this construction, Broué, Malle and Rouquier \cite{BMR} defined such a deformation for every complex reflection group, also known as Hecke algebra. Besides, Ariki and Koike \cite{ArKo} defined such a Hecke algebra $\H_n(q, \tuple{u})$  for $G(r, 1, n)$ where $q$ and $\tuple{u} = (u_1, \dots, u_r)$ are some parameters, followed by Ariki \cite{Ar_rep} who did the same thing for $G(r, p, n)$. In particular, for a suitable choice of parameters $q$ and $\tuple{u}$, this Hecke algebra $\H_{p, n}^{\tuple{\Lambda}}(q)$ of $G(r, p, n)$ can be seen as a subalgebra of $\H_n(q, \tuple{u})$.

In the semisimple case, Ariki and Koike have determined all irreducible modules for $\H_n(q, \tuple{u})$. The modular case was treated by Ariki and Mathas \cite{ArMa, Ar_class}, and also by Graham and Lehrer~\cite{GrLe} and Dipper, James and Mathas~\cite{DJM}, using the theory of cellular algebras. Moreover, in  \cite{Ar_dec}, Ariki proved a conjecture of Lascoux, Leclerc and Thibon \cite{LLT}. The theorem has the following consequence in characteristic $0$. 
If all $u_k$ for $1 \leq k \leq r$ are powers of $q$,
then determining the decomposition matrix of $\H_n(q, \tuple{u})$ or the canonical basis of a certain integrable highest weight $\widehat{\mathfrak{sl}}_e$-module $L(\tuple{\Lambda})$ are equivalent problems, where $\widehat{\mathfrak{sl}}_e$ denotes the Kac--Moody algebra of type $A_{e-1}^{(1)}$.
Together with the work of Uglov \cite{Ug} which computes this canonical basis, we are thus able to explicitly describe the decomposition matrix of $\H_n(q, \tuple{u})$. 

Once again in the semisimple case, Ariki \cite{Ar_rep} used Clifford theory to determine all irreducible modules for $\H_{p, n}^{\tuple{\Lambda}}(q)$. In the modular case, Genet and Jacon~\cite{GeJa} and Chlouveraki and Jacon~\cite{ChJa} gave a parametrisation of the simple modules of $\H_{p, n}^{\tuple{\Lambda}}(q)$ over $\mathbb{C}$, and  Hu~\cite{Hu_modular,Hu_crystal}  classified them over a field containing a primitive $p$th root of unity.
Furthermore, Hu and Mathas~\cite{HuMa_Mo,HuMa_dec} gave a procedure to compute the decomposition matrix of $\H_{p, n}^{\tuple{\Lambda}}(q)$ in characteristic $0$, under a separation condition (where the Hecke algebra is not semisimple in general). Let us also mention the work of Geck~\cite{Ge}, who deals with the case of type $D$ (corresponding to $r = p = 2$).

Partially motivated by Ariki's theorem, Khovanov and Lauda \cite{KhLau1, KhLau2} and Rouquier \cite{Rou} independently introduced the algebra $\qH_n(\Gamma)$, known as a \emph{quiver Hecke algebra} or \emph{KLR algebra}. This led to a categorification result:
\[
U_q^-(\mathfrak{g}_\Gamma) \simeq \bigoplus_{n \geq 0} \left[\mathrm{Proj}(\qH_n(\Gamma))\right],
\]
where $U_q^-(\mathfrak{g}_\Gamma)$ is the negative part of the quantum group of $\mathfrak{g}_\Gamma$, the Kac--Moody algebra associated with the quiver $\Gamma$, and $ [\mathrm{Proj}(\qH_n(\Gamma))]$ denotes the Grothendieck group of the additive category of finitely generated graded projective $\qH_n(\Gamma)$-modules.
 Moreover, considering some cyclotomic quotients $\qH_n^{\tuple{\Lambda}}(\Gamma)$ of the quiver Hecke algebra, Kang and Kashiwara~\cite{KanKa} also proved a categorification result for the highest weight $U_q(\mathfrak{g}_\Gamma)$-modules.

With $\Gamma$ a quiver of type $A_{e-1}^{(1)}$ and a suitable choice of $\tuple{u}$, we thus obtain a connection between the Ariki--Koike algebra $\H_n(q, \tuple{u})$ and $\qH_n^{\tuple{\Lambda}}(\Gamma)$.  Brundan and Kleshchev \cite{BrKl} (and independently Rouquier \cite{Rou}) gave an explicit isomorphism between these two algebras. In particular, the Ariki--Koike algebra inherits the natural $\mathbb{Z}$-grading of $\qH_n^{\tuple{\Lambda}}(\Gamma)$. Further, Hu and Mathas \cite{HuMa_gra} constructed a homogeneous (cellular) basis of the Ariki--Koike algebra.

We aim to generalise the previous results concerning $G(r,1,n)$ to the remaining elements of the infinite series $\{G(r,p,n)\}$ (and thus for all but finitely many complex reflection groups). In this paper, we restrict the isomorphism of \cite{BrKl} to the subalgebra $\H_{p, n}^{\tuple{\Lambda}}(q)$ of $\H_n(q, \tuple{u})$ (for a suitable choice of $\tuple{u}$), and then study its image in the cyclotomic quiver Hecke algebra $\qH_n^{\tuple{\Lambda}}(\Gamma)$. Our two main results are the following.
\begin{itemize}
\item We obtain a cyclotomic quiver Hecke-like presentation for $\H_{p, n}^{\tuple{\Lambda}}(q)$ (see Corollary~\ref{corollary:presentation_H(G(rpn))}).
\item We find an isomorphism which makes explicit the Morita equivalence of Ariki--Koike algebras of \cite{DiMa} (see \textsection\ref{subsection:unexpected_corollary}).
\end{itemize}
We will first need to generalise the main result of \cite{BrKl}. Surprisingly, combined with a theorem of \cite{Ro}, this leads to the isomorphism realising the above Morita equivalence. We then exploit the fact that $\H_{p, n}^{\tuple{\Lambda}}(q)$ is the fixed point subalgebra of $\H_n(q, \tuple{u})$ for a particular automorphism $\sigma$, the \emph{shift automorphism}. After some technical work, we obtain the cyclotomic quiver Hecke-like presentation for $\H_{p, n}^{\tuple{\Lambda}}(q)$. We deduce that $\H_{p, n}^{\tuple{\Lambda}}(q)$ depends only on the quantum characteristic and is a graded subalgebra of $\H_n(q, \tuple{u})$.
We note that Boys and Mathas \cite{Bo, BoMa} already studied a restriction of the isomorphism of \cite{BrKl} to the fixed point subalgebra of an automorphism, in order to study \emph{alternating} (cyclotomic) Hecke algebras. The approach taken here is globally similar, however, we are here able to use a trick of Stroppel and Webster~\cite{StWe} to simplify the final proof (see \textsection\ref{subsection:nice_family}).

We now give a brief overview of this article.
Let $r, p, d,n \in \mathbb{N}^*$ be some integers with $r = d p$, an element $q \neq 0, 1$ of a field $F$, a primitive $p$th root of unity $\zeta \in F^\times$ and $J \coloneqq \mathbb{Z}/p\mathbb{Z} \simeq \langle \zeta \rangle$. We set $e \in \mathbb{N}_{\geq 2} \cup \{\infty\}$ the order of $q$ in $F^\times$ and we define $I \coloneqq \mathbb{Z}/e\mathbb{Z}$ (with $I \coloneqq \mathbb{Z}$ if $e = \infty$). Finally, we consider a tuple $\tuple{\varLambda} = (\varLambda_{i, j})$ where $(i, j) \in I \times J$ and $\varLambda_{i, j} \in \mathbb{N}$.
We begin Section~\ref{section:AK} by defining in  \textsection\ref{subsection:Gr1n} the  Hecke algebra $\sH_n(q, \tuple{u})$ of type $G(r,1,n)$ (the Ariki--Koike algebra), which we write $\sH_n^{\tuple{\varLambda}}(q, \zeta)$ when each $u_k$ is of the form $\zeta^j q^i$. The algebra $\sH_n^{\tuple{\varLambda}}(q, \zeta)$ is generated by some elements $S, T_1, \dots, T_{n-1}$, subject to relations \eqref{relation:hecke_ordre}--\eqref{relation:hecke_braid3} and the ``cyclotomic'' one:
\[
\prod_{i \in I} \prod_{j \in J} {(S - \zeta^j q^i)}^{\varLambda_{i, j}} = 0
\]
(see \eqref{relation:hecke_cyclo_S}).
We define in Proposition~\ref{proposition:G(r1n)-def_sigma} an important object of this paper, the \emph{shift automorphism} $\sigma$ of $\H_n^{\tuple{\varLambda}}(q, \zeta)$: it maps $S$ to $\zeta S$ and is the identity on the remaining generators $T_1, \dots, T_{n-1}$.
We then start \textsection\ref{subsection:G(rpn)} by defining the cyclotomic Hecke algebra $\H_{p, n}^{\tuple{\varLambda}}(q)$ of type $G(r,p,n)$. Our definition differs from Ariki's \cite{Ar_rep}. However, in Appendix~\ref{section:appendix_induction}, after we gave a short proof that $\H_{1, n}^{\tuple{\varLambda}}(q)$ is the usual Ariki--Koike algebra, we prove that if $p \geq 2$ then our definition is equivalent to Ariki's. In particular, this shows that Ariki indeed defined  a Hecke algebra of $G(r,p,n)$ as defined in \cite{BMR} (this fact is mentionned in \cite{BMR} but we did not find any proof in the literature).
We then prove in Corollary~\ref{corollary:H(G(rpn))_fixe} that $\H_{p, n}^{\tuple{\varLambda}}(q)$ is the fixed point subalgebra of $\H_n^{\tuple{\varLambda}}(q, \zeta)$ under the shift automorphism.
We introduce in \textsection\ref{subsection:removing_repetitions} a divisor $p'$ of $p$, together with the set $J' \coloneqq \{1, \dots, p'\}$, such that the map $I \times J' \ni (i, j) \mapsto \zeta^j q^i$ is one-to-one and has the same image as $I \times J \ni (i, j) \mapsto \zeta^j q^i$. We then define a finitely-supported tuple $\tuple{\Lambda}$,  indexed by $I \times J'$, associated with the tuple $\tuple{\varLambda} \in \mathbb{N}^{(I \times J)}$ (see Proposition~\ref{proposition:removing_repetitions-_Lambda}). We also introduce the notation $\H_n^{\tuple{\Lambda}}(q, \zeta)$ and $\H_{p, n}^{\tuple{\Lambda}}(q)$. The reader should not be afraid of confusing the two notations $\tuple{\varLambda}$ and $\tuple{\Lambda}$: we will not use $\tuple{\varLambda}$ after Section~\ref{section:AK}.

Now let $\Gamma$ be a loop-free quiver with no repeated edges. Let $K$ be the vertex set of $\Gamma$ and let $\tuple{\Lambda}$ be a tuple of non-negative integers indexed by $K$. We define in Section~\ref{section:CQHA} the quiver Hecke algebra $\sqH_n(\Gamma)$ and its cyclotomic quotient $\sqH_n^{\tuple{\Lambda}}(\Gamma)$.
In \textsection\ref{subsection:fixed_subalgebra}, given a permutation of the vertices of $\Gamma$, we associate in Theorem~\ref{theorem:definition_sigma_quiver} an  automorphism of $\sqH_n(\Gamma)$ and we easily give in \textsection\ref{subsubsection:fixed_affine_subalgebra} a presentation of the fixed point subalgebra (Corollary~\ref{corollary:presentation_quiver_fixed}). In contrast, we need a little bit more work in \textsection\ref{subsubsection:fixed_cyclotomic_subalgebra} to do the same thing for the cyclotomic quotient $\sqH_n^{\tuple{\Lambda}}(\Gamma)$. We give a presentation for the fixed point subalgebra in Theorem~\ref{theorem:quiver_two_subalgebra_same}. Note that  characteristic-free statements can be found in the author's PhD thesis~\cite[\textsection 1.4]{Ro_phd}.

 We generalise in Section~\ref{section:BK} the main result of \cite{BrKl}. The calculations are entirely similar, hence we do not write them down. More precisely, we prove the following $F$-isomorphism (Theorem~\ref{theorem:BK_generalised}):
\[
\sH_n(q, \tuple{u}) \simeq \sqH_n^{\tuple{\Lambda}}(\Gamma).
\]
The quiver $\Gamma$ is given by $p'$ copies of the cyclic quiver $\Gamma_e$ with $e$ vertices (which is a two-sided infinite line if $e = \infty$), where $\tuple{u}$ is such that the set $\{u_1, \dots, u_r\}$ is a union of $p'$ orbits for the action of $\langle q \rangle$ on $F^\times$. In particular, in the setting of \cite{BrKl} we have $p' = 1$. Moreover, we deduce that  this isomorphism realises the well-known Morita equivalence of \cite{DiMa} involving Ariki--Koike algebras, see \textsection\ref{subsection:unexpected_corollary}.

In Section~\ref{section:presentation}, we show that the isomorphism of Theorem~\ref{theorem:BK_generalised} can be chosen such that the shift automorphism of $\H_n^{\tuple{\Lambda}}(q, \zeta)$
corresponds to a nice automorphism of $\qH_n^{\tuple{\Lambda}}(\Gamma)$. This automorphism is built from the automorphism of $\Gamma$ which maps a vertex $\vertex{v} = \zeta^j q^i$ for $(i, j) \in K \coloneqq I \times J'$ to $\sigma(\vertex{v}) \coloneqq \zeta \vertex{v}$. This explains why we chose the term ``shift automorphism''. We finally deduce with Corollary~\ref{corollary:presentation_H(G(rpn))} a cyclotomic quiver Hecke-like presentation for $\H_{p, n}^{\tuple{\Lambda}}(q)$. In particular, this implies that $\H_{p, n}^{\tuple{\Lambda}}(q)$ is a graded subalgebra of $\H_n^{\tuple{\Lambda}}(q, \zeta)$ (Corollary~\ref{corollary:intertwining-graded_subalgebra}) and  that $\H_{p, n}^{\tuple{\Lambda}}(q)$ does not depend on $q$ but on the quantum characteristic $e$ (Corollary~\ref{corollary:intertwining-independence_q}).

\paragraph*{Acknowledgements} I am grateful to Maria Chlouveraki, Nicolas Jacon, Noah White and the referee for their numerous corrections. I also thank Andrew Mathas for his comments on a preliminary version of this paper, and Ivan Marin for discussions concerning the presentation of $G(r,p,n)$.

\paragraph*{Notation} Let $d, p, n \in \mathbb{N}^*$ with $n \geq 2$ and set $r \coloneqq pd$. We work over a field $F$ that contains a primitive $p$th root of unity $\zeta$, in particular, the characteristic of $F$ does not divide $p$. We consider an element $q \in F \setminus \{0, 1\}$ and we write $e \in \mathbb{N}_{\geq 2} \cup\{\infty\}$ its order in $F^\times$. We define:
\[
I \coloneqq \begin{cases}
\mathbb{Z} / e\mathbb{Z} & \text{if } e < \infty,
\\
\mathbb{Z} & \text{otherwise } (e = \infty),
\end{cases}
\]
and $J \coloneqq \mathbb{Z} /p\mathbb{Z} \simeq \langle \zeta \rangle$.

\section{The Hecke algebra of \texorpdfstring{$G(r,p,n)$}{G(r,p,n)}}
\label{section:AK}

Here we review two constructions of Ariki and Koike. Let $\tuple{u} = (u_1, \dots, u_r)$ be a tuple of elements of $F^\times$.

\subsection{The Hecke algebra of \texorpdfstring{$G(r,1,n)$}{G(r,1,n)}}
\label{subsection:Gr1n}

We recall here the definition of the cyclotomic Hecke algebra $\sH_n(q, \tuple{u})$ of type $G(r,1,n)$, also known as Ariki--Koike algebra.

\begin{definition}[\cite{ArKo}]
The algebra $\sH_n(q, \tuple{u})$ is the unitary associative $F$-algebra generated by the elements $S, T_1, \dots, T_{n-1}$, subject to the following relations:
\begin{subequations}
\begin{align}
\label{relation:hecke_cyclo_S_general}
\prod_{k = 1}^r (S - u_k)  &= 0,
\\
\label{relation:hecke_ordre}
\quad (T_a + 1)(T_a - q) &= 0,
\\
\label{relation:hecke_ST1ST1}
S T_1 S T_1 &= T_1 S T_1 S,
\\
\label{relation:hecke_STa}
S T_a &= T_a S \quad \text{if } a > 1,
\\
\label{relation:hecke_braid2}
T_a T_b &= T_b T_a \quad \text{if } |a-b| > 1,
\\
\label{relation:hecke_braid3}
T_a T_{a+1} T_a &= T_{a+1} T_a T_{a+1}.
\end{align}
\end{subequations}
\end{definition}


According to \cite{BMR}, the algebra $\sH_n(q, \tuple{u})$ is a Hecke algebra of the complex reflection group $G(r, 1, n)$.
Let $X_1 \coloneqq S$ and define for $a \in \{1, \dots, n-1\}$ the elements $X_{a+1} \in \sH_n(q, \tuple{u})$ by:
\begin{equation}
\label{equation:definition_Xa+1}
q X_{a+1} \coloneqq T_a X_a T_a.
\end{equation}
These elements $X_1, \dots, X_n$ pairwise commute (\cite[Lemma 3.3.(2)]{ArKo}).
Moreover, Matsumoto's theorem (see, for instance, \cite[Theorem 1.2.2]{GePf}) ensures that \eqref{relation:hecke_braid2} and \eqref{relation:hecke_braid3} allow us to define $T_w \coloneqq T_{a_1} \cdots T_{a_m}$ for any reduced expression $w = s_{a_1} \cdots s_{a_m} \in \mathfrak{S}_n$, where $s_a \in \mathfrak{S}_n$ is the transposition $(a, a+1)$.

\begin{theoreme}[\protect{\cite[Theorem 3.10]{ArKo}}]
\label{theorem:basis_AK}
The elements
\begin{equation}
\label{equation:basis_AK}
X_1^{m_1} \cdots X_n^{m_n} T_w
\end{equation}
for $m_1, \dots, m_n \in \{0, \dots, r - 1\}$ and  $w \in \mathfrak{S}_n$
form a basis of the $F$-vector space $\sH_n(q, \tuple{u})$.
\end{theoreme}

Let $\tuple{\varLambda} = (\varLambda_{i, j}) \in \mathbb{N}^{(I \times J)}$ be a finitely-supported tuple of non-negative integers with $\sum_{i \in I} \sum_{j \in J} \varLambda_{i, j} = r$. We say that $\tuple{\varLambda}$ is a \emph{weight of level $r$}. In particular, we can choose the parameters $u_1, \dots, u_r$ such that the relation \eqref{relation:hecke_cyclo_S_general} in $\sH_n(q, \tuple{u})$ becomes the following one:
\begin{equation}
\label{relation:hecke_cyclo_S}
\prod_{i \in I} \prod_{j \in J} {(S - \zeta^j q^i)}^{\varLambda_{i, j}} = 0.
\end{equation}

\begin{definition}
In the above setting, we define $\sH_n^{\tuple{\varLambda}}(q, \zeta) \coloneqq \sH_n(q, \tuple{u})$.
\end{definition}

We will often need the following condition on $\tuple{\varLambda}$:
\begin{equation}
\label{equation:G(r1n)-condition_varLambda}
\varLambda_{i, j} = \varLambda_{i, j'} \eqqcolon \varLambda_i
\qquad \text{for all } i \in I \text{ and } j, j' \in J.
\end{equation}
In this case, the weight $(\varLambda_i)_{i \in I}$ has level $d = \frac{r}{p}$. Moreover, we can write \eqref{relation:hecke_cyclo_S} as:
\[
\prod_{i \in I} \prod_{j \in J} {(S - \zeta^j q^i)}^{\varLambda_i} = \prod_{i \in I} {(S^p - q^{pi})}^{\varLambda_i} = 0.
\]
Thus, we get the following result.

\begin{proposition}
\label{proposition:G(r1n)-def_sigma}
Suppose that $\tuple{\varLambda}$ satisfies \eqref{equation:G(r1n)-condition_varLambda}.
There is a well-defined algebra homomorphism $\sigma : \H_n^{\tuple{\varLambda}}(q, \zeta) \to \H_n^{\tuple{\varLambda}}(q, \zeta)$ given by:
\begin{align*}
\sigma(S) &\coloneqq \zeta S,
\\
\sigma(T_a) &\coloneqq T_a, \qquad \text{for all } a \in \{1, \dots, n-1\}.
\end{align*}
\end{proposition}

\begin{remarque}
The homomorphism $\sigma$ has order $p$, in particular $\sigma$ is bijective. We will refer to $\sigma$ as the \emph{shift automorphism} of $\H_n^{\tuple{\varLambda}}(q, \zeta)$.
\end{remarque}

In the remaining part of this subsection, we assume that \eqref{equation:G(r1n)-condition_varLambda} is satisfied, so that the shift automorphism is defined.
The following lemma is an easy induction.
\begin{lemme}
\label{lemma:action_sigma_Xa}
For every $a \in \{1, \dots, n\}$ we have $\sigma(X_a) = \zeta X_a$.
\end{lemme}

\begin{proposition}
\label{proposition:basis_fixed_point_Hecke}
The elements of $\H_n^{\tuple{\varLambda}}(q, \zeta)$ fixed by $\sigma$ are exactly the elements in the $F$-span of $X_1^{m_1} \cdots X_n^{m_n} T_w$ for $m_1, \dots, m_n \in \{0, \dots, r-1\}$ and  $w \in \mathfrak{S}_n$, with the additional following condition:
\[
m_1 + \dots + m_n = 0 \pmod{p}.
\]
\end{proposition}

\begin{proof}
Let $h$ be an arbitrary element of $\H_n^{\tuple{\varLambda}}(q, \zeta)$. By Theorem~\ref{theorem:basis_AK}, we can write, with $\tuple{m} = {(m_a)_a}$:
\[
h = \sum_{\substack{\tuple{m} \in \mathbb{N}^n, w \in \mathfrak{S}_n \\ 0 \leq m_a < r}}
h_{\tuple{m}, w} X_1^{m_1} \cdots X_n^{m_n} T_w,
\]
for some $h_{\tuple{m}, w} \in F$. Applying Lemma~\ref{lemma:action_sigma_Xa}, we have:
\[
\sigma(h) = \sum_{\substack{\tuple{m} \in \mathbb{N}^n, w \in \mathfrak{S}_n \\ 0 \leq m_a < r}}
h_{\tuple{m}, w} \zeta^{m_1 + \dots + m_n} X_1^{m_1} \cdots X_n^{m_n} T_w,
\]
thus $\sigma(h) = h$ if and only if $\zeta^{m_1 + \dots + m_n} = 1$ when $h_{\tuple{m}, w} \neq 0$. We conclude since $\zeta$ is a primitive $p$th root of unity.
\end{proof}

Note that the family in Proposition~\ref{proposition:basis_fixed_point_Hecke} is a basis, since it is free (by Theorem~\ref{theorem:basis_AK}).

\subsection{The Hecke algebra of \texorpdfstring{$G(r,p,n)$}{G(r,p,n)}}
\label{subsection:G(rpn)}

We continue to assume here that the weight $\tuple{\varLambda}$ satisfies the condition \eqref{equation:G(r1n)-condition_varLambda}. In particular, for any $i \in I$ and $j \in J$ we have $\varLambda_{i, j} = \varLambda_i$. We will first define the algebra that Ariki~\cite{Ar_rep} associated with $G(r,p,n)$, and then relate this algebra to \textsection\ref{subsection:Gr1n}.

\begin{definition}[\cite{Ar_rep}]
\label{definition:hecke_G(rpn)}
We denote by $\H_{p, n}^{\tuple{\varLambda}}(q)$ the unitary associative $F$-algebra generated by $s, t'_1, t_1, \dots, t_{n-1}$, subject to the following relations:
\begin{subequations}
\begin{align}
\label{relation:hecke_G(rpn)_cyclo_s}
\prod_{i \in I} {(s - q^{pi})}^{\varLambda_i} &= 0,
\\
\label{relation:hecke_G(rpn)_ordre_ta}
(t'_1 + 1)(t'_1 - q) &= (t_a + 1)(t_a - q) = 0,
\\
t'_1 t_2 t'_1 &= t_2 t'_1 t_2, \quad t_a t_{a+1} t_a = t_{a+1} t_a t_{a+1},
\\
\label{relation:hecke_G(rpn)_braid6}
(t'_1 t_1 t_2)^2 &= (t_2 t'_1 t_1)^2,
\\
t'_1 t_a &= t_a t'_1 \quad \text{if } a \in \{3, \dots, n-1\},
\\
t_a t_b &= t_b t_a \quad \text{if } |a-b| > 1,
\\
\label{relation:hecke_G(rpn)_sta}
s t_a &= t_a s \quad \text{if } a \in \{2, \dots, n-1\},
\\
\label{relation:hecke_G(rpn)_braid3}
s t'_1 t_1 &= t'_1 t_1 s,
\\
\label{relation:hecke_G(rpn)_bigbraid}
\underbrace{s t'_1 t_1 t'_1 t_1 \dots}_{p + 1 \text{ factors}} &= \underbrace{t_1 s t'_1 t_1 t'_1 \dots}_{p + 1 \text{ factors}}.
\end{align}
\end{subequations}
\end{definition}

According to \cite{BMR}, the algebra $\H_{p, n}^{\tuple{\varLambda}}(q)$ is a Hecke algebra of the complex reflection group $G(r,p,n)$.

\begin{remarque}
\label{remark:G(rpn)_case_p=1}
We prove in \textsection\ref{subsection:appendix-p1} that the above presentation considered for $p = 1$ yields the Ariki--Koike algebra $\H_n^{\tuple{\varLambda}}(q, 1)$ as defined in \textsection\ref{subsection:Gr1n}.
\end{remarque}

\begin{remarque}
\label{remark:hecke_G(rpn)-equivalent_presentation}
Assume here that $p \geq 2$. 
The reader may have noticed that the above presentation is not the one given by Ariki~\cite{Ar_rep}. Instead of \eqref{relation:hecke_G(rpn)_bigbraid} Ariki gives the following relation:
\begin{equation}
\label{relation:hecke_G(rpn)_strange}
s t'_1 t_1 = {(q^{-1} t'_1 t_1)}^{2-p} t_1 s t'_1 + (q-1) \sum_{k = 1}^{p-2} {(q^{-1} t'_1 t_1)}^{1-k} s t'_1.
\tag{\ref*{relation:hecke_G(rpn)_bigbraid}'}
\end{equation}
We claim that these two presentations define isomorphic algebras. Using \eqref{relation:hecke_G(rpn)_ordre_ta}, we can show by induction the following equality:
\[
{(q^{-1} t'_1 t_1)}^{2-p} t_1 s t'_1 + (q-1) \sum_{k = 1}^{p-2} {(q^{-1} t'_1 t_1)}^{1-k} s t'_1
=
(\underbrace{t_1^{-1} {t'_1}^{-1} t_1^{-1} {t'_1}^{-1}\dots}_{p-2 \text{ factors}}) (\underbrace{\dots t_1 t'_1 t_1 t'_1}_{p - 2 \text{ factors}}) t_1 s t'_1.
\]
Hence, using \eqref{relation:hecke_G(rpn)_braid3} we get that \eqref{relation:hecke_G(rpn)_bigbraid} and \eqref{relation:hecke_G(rpn)_strange} are equivalent, which proves the claim (we refer to \textsection\ref{subsection:appendix-two_equivalent_relations} for more details). We conclude this remark by mentioning that the generators of \cite{Ar_rep} are given by $a_0 = s, a_1 = t'_1$ and $a_k = t_{k-1}$ for $k \in \{2, \dots, n\}$.
\end{remarque}

We state the main result of this section (note that the case $p = 1$ is proved by Theorems~\ref{theorem:basis_AK} and \ref{theorem:appendix-case_p=1}).

\begin{theoreme}[\protect{\cite[Proposition 1.6]{Ar_rep}}]
\label{theorem:homomorphism_Ariki}
The algebra homomorphism $\phi : \H_{p, n}^{\tuple{\varLambda}}(q) \to \H_n^{\tuple{\varLambda}}(q, \zeta)$ given by:
\begin{align*}
\phi(s) &\coloneqq S^p,
\\
\phi(t'_1) &\coloneqq S^{-1} T_1 S,
\\
\phi(t_a) &\coloneqq T_a, \qquad \text{for all } a \in \{1, \dots, n-1\}.
\end{align*}
is well-defined and one-to-one. Moreover, the elements $X_1^{m_1} \cdots X_n^{m_n} T_w$ of Theorem~\ref{theorem:basis_AK} such that $m_1 + \dots + m_n = 0 \pmod{p}$ form an $F$-basis of $\phi(\H_{p, n}^{\tuple{\varLambda}}(q))$.
\end{theoreme}

In particular, using Proposition~\ref{proposition:basis_fixed_point_Hecke} we get the following one.

\begin{corollaire}
\label{corollary:H(G(rpn))_fixe}
The algebra $\H_{p, n}^{\tuple{\varLambda}}(q)$ is isomorphic via $\phi$ to $\H_n^{\tuple{\varLambda}}(q, \zeta)^\sigma$, the fixed point subalgebra of $\H_n^{\tuple{\varLambda}}(q, \zeta)$ under the shift automorphism $\sigma$.
\end{corollaire}

\subsection{Removing repetitions}
\label{subsection:removing_repetitions}

The following map:
\[
\begin{array}{|rcl}
I \times J & \longrightarrow & F^\times
\\
(i, j) & \longmapsto & \zeta^j q^i
\end{array},
\]
is not one-to-one. The first aim of this subsection is to find a subset $J' \subseteq \{1, \dots, p\} \simeq J$ such that the restriction of the previous map to $I \times J'$ has the same image and is one-to-one. Moreover, for our purposes, we would like relation \eqref{relation:hecke_cyclo_S} to be of the form:
\begin{equation}
\label{relation:removing_repetitions-would_like}
\prod_{i \in I} \prod_{j \in J'} {(S - \zeta^j q^i)}^{\Lambda_{i, j}} = 0,
\end{equation}
where $\tuple{\Lambda} = (\Lambda_{i, j})_{i \in I, j \in J'} \in \mathbb{N}^{(I \times J')}$ is a weight of level $r$. The second aim of this subsection is to know for which tuples $\tuple{\Lambda} \in \mathbb{N}^{(I \times J')}$ of level $r$ there is some $\tuple{\varLambda} \in \mathbb{N}^{(I \times J)}$ such that the relation \eqref{relation:hecke_cyclo_S} in $\H_n^{\tuple{\varLambda}}(q, \zeta)$ is exactly \eqref{relation:removing_repetitions-would_like}. We will be particularly interested in the case where $\tuple{\varLambda}$ satisfies \eqref{equation:G(r1n)-condition_varLambda}. This will require some quite long but easy computations.

Let us define the following integer:
\begin{equation}
\label{equation:definition_p'}
p' \coloneqq \min\{m \in \mathbb{N}^* : \zeta^m \in \langle q \rangle\} \in \{1, \dots, p\},
\end{equation}
together with the following set:
\[
J' \coloneqq \{1, \dots, p'\}.
\]

\begin{lemme}
\label{lemma:introduction-p'_gcd}
The integer $p'$ is given by:
\begin{itemize}
\item if $e = \infty$ then $p' = p$;
\item if $e < \infty$ then  $p' = \frac{p}{\mathrm{gcd}(p, e)}$.
\end{itemize}
In particular, the integer $p'$ divides $p$ and depends only on $p$ and $e$.
\end{lemme}

\begin{proof}
The statement for $e = \infty$ is obvious since each element of $\langle q \rangle \setminus \{1\}$ has infinite order. Thus, we now assume that $e < \infty$.
For any $m \in \mathbb{N}^*$, the order of $\zeta^m$ in $F^\times$ is $\frac{p}{\mathrm{gcd}(p,m)}$. Since $q$ is a primitive $e$th root of unity, the set $\langle q \rangle$ is precisely the set of elements of $F^\times$ of order dividing $e$. Hence:
\begin{align*}
\zeta^m \in \langle q \rangle
&\iff
\frac{p}{\mathrm{gcd}(p, m)} \text{ divides } e
\\
&\iff
\frac{p}{\mathrm{gcd}(p, m)} \text{ divides } \mathrm{gcd}(p, e)
\\
&\iff
\frac{p}{\mathrm{gcd}(p, e)} \text{ divides } \mathrm{gcd}(p, m).
\end{align*}
We conclude that the minimal $m \in \mathbb{N}^*$ such that $\zeta^m \in \langle q \rangle$ is $p' =  \frac{p}{\mathrm{gcd}(p,e)}$.
\end{proof}

The first aim of this subsection is achieved thanks to the next lemma, which is a immediate consequence of the minimality of $p'$.

\begin{lemme}
\label{lemma:zeta^j_q^i_distinct}
The elements $\zeta^j q^i$ for $i \in I$ and $j \in J'$ are pairwise distinct.
\end{lemme}


Let us denote by $\eta$ the (unique) element of $I$ such that:
\begin{equation}
\label{equation:introduction-zetap'_qeta}
\zeta^{p'} = q^\eta.
\end{equation}
Note that $p' = p \iff \eta = 0 \iff \langle q \rangle \cap \langle \zeta \rangle = \{1\}$. In particular, if $\eta \neq 0$ then $e < \infty$. In that case, we  are not necessarily in the setting of \cite{HuMa_dec} (see [Lemma~$2.6.(a)$, \textit{loc. cit.}]). 
We now consider the  following map:
\begin{equation}
\label{equation:map_J'_ZomegaZ_J}
\begin{array}{|rcl}
J' \times \mathbb{Z}/\omega\mathbb{Z} & \longrightarrow & J
\\
(j, a) & \longmapsto & j + p' a
\end{array},
\end{equation}
where $\omega \coloneqq \frac{p}{p'}$.
It is well-defined and surjective, hence bijective by a counting argument. Equation \eqref{relation:hecke_cyclo_S} becomes:
\begin{align}
\prod_{i \in I} \prod_{j \in J} {(S - \zeta^j q^i)}^{\varLambda_{i, j}}
&=
\prod_{i \in I} \prod_{j \in J'} \prod_{a \in \mathbb{Z}/\omega\mathbb{Z}} {\left(S - \zeta^j (\zeta^{p'})^a q^i\right)}^{\varLambda_{i, j + p'a}}
\notag
\\
&=
\prod_{i \in I} \prod_{j \in J'} \prod_{a\in \mathbb{Z}/\omega\mathbb{Z}} {\left(S - \zeta^j q^{i + \eta a}\right)}^{\varLambda_{i, j + p'a}}.
\label{equation:removing_repetitions-construction_Lambda}
\end{align}
For each $(i, j) \in I \times J'$, we define:
\begin{equation}
\label{equation:removing_repetitions-definition_Lambda}
\Lambda_{i, j} \coloneqq \sum_{i' \in I}  \sum_{\substack{a \in \mathbb{Z}/\omega\mathbb{Z} \\ i' + \eta a = i}} \varLambda_{i', j + p'a},
\end{equation}
so that, by \eqref{equation:removing_repetitions-construction_Lambda}:
\[
\prod_{i \in I} \prod_{j \in J} {(S - \zeta^j q^i)}^{\varLambda_{i, j}} =
\prod_{i \in I} \prod_{j \in J'} {(S - \zeta^j q^i)}^{\Lambda_{i, j}}.
\]
Hence, relation \eqref{relation:hecke_cyclo_S} transforms to the desired one \eqref{relation:removing_repetitions-would_like}. Conversely, it is clear that each weight $\tuple{\Lambda} \in \mathbb{N}^{(I \times J')}$ comes from some $\tuple{\varLambda} \in \mathbb{N}^{(I \times J)}$ through \eqref{equation:removing_repetitions-definition_Lambda}, that is, for each $\tuple{\Lambda} \in \mathbb{N}^{(I \times J')}$ there is some $\tuple{\varLambda} \in \mathbb{N}^{(I \times J)}$ such that \eqref{equation:removing_repetitions-definition_Lambda} is satisfied. Indeed, given any $\tuple{\Lambda} \in \mathbb{N}^{(I \times J')}$ it suffices to set:
\[
\varLambda_{i, j} \coloneqq
\begin{cases}
\Lambda_{i, \jmath} &\text{if }  j \text{ is the image of } (\jmath,0) \text{ by the bijection of } \eqref{equation:map_J'_ZomegaZ_J},
\\
0 & \text{otherwise,}
\end{cases}
\]
for all $(i, j) \in I \times J$.

\begin{definition}
\label{definition:removing_definition-removing_repetitions_sH}
Let $\tuple{\Lambda} \in \mathbb{N}^{(I \times J')}$ be a weight of level $r$. We consider $\tuple{\varLambda} \in \mathbb{N}^{(I \times J)}$ a weight of level $r$ which gives $\tuple{\Lambda}$ through \eqref{equation:removing_repetitions-definition_Lambda}.
We write $\sH_n^{\tuple{\Lambda}}(q, \zeta) \coloneqq \sH_n^{\tuple{\varLambda}}(q, \zeta)$. In particular, the relation \eqref{relation:hecke_cyclo_S} becomes  \eqref{relation:removing_repetitions-would_like}.
\end{definition}

We now assume that the weight $\tuple{\varLambda} \in \mathbb{N}^{(I \times J)}$ satisfies the condition \eqref{equation:G(r1n)-condition_varLambda}, that is, factors to $\tuple{\varLambda} \in \mathbb{N}^{(I)}$: we want to know which condition we recover on $\tuple{\Lambda}$. The defining equality \eqref{equation:removing_repetitions-definition_Lambda} becomes:
\[
\Lambda_{i, j} = \sum_{i' \in I} \sum_{\substack{a \in \mathbb{Z}/\omega\mathbb{Z} \\ i' + \eta a = i}} \varLambda_{i'},
\]
for $i \in I$ and $j \in J'$.
In particular, for any $i \in I$ and $j, j' \in J'$ we have $\Lambda_{i, j} = \Lambda_{i, j'} \eqqcolon \Lambda_i$, so that $\tuple{\Lambda} \in \mathbb{N}^{(I)}$ is a weight of level $\omega d$, and:
\begin{equation}
\label{equation:removing_repetitions-definition_Lambda_w/o_j}
\Lambda_i = \sum_{i' \in I} \sum_{\substack{a \in \mathbb{Z}/\omega\mathbb{Z} \\ i' + \eta a = i}} \varLambda_{i'},
\end{equation}
for all $i \in I$.

\begin{lemme}
\label{lemma:remove_repetitions}
For any $i \in I$ we have:
\begin{numcases}
{\#\{a \in \mathbb{Z}/\omega\mathbb{Z} : \eta a = i\} =}
\label{equation:lemma_choice_parameters_notin}
0 & if  $i \notin \eta I$,
\\
\label{equation:lemma_choice_parameters_in}
1 & if  $i \in \eta I$.
\end{numcases}
\end{lemme}

\begin{proof}
The result is straightforward if $\eta = 0$, in particular in that case we have $\omega = 1$. Thus we assume $\eta \neq 0$, in particular $e < \infty$ and $I = \mathbb{Z}/e\mathbb{Z}$.
Let us compute the cardinality of the fibre of $i$ under the following group homomorphism:
\[
\begin{array}{c|rcl}
\phi : &\mathbb{Z} & \longrightarrow & I
\\
&a & \longmapsto & \eta a.
\end{array}
\]
First, the image of $\phi$ is $\eta I$; this proves \eqref{equation:lemma_choice_parameters_notin}. The element $\omega$ lies in $\ker \phi$. Indeed, we have $\mathrm{order}(\zeta^{p'}) = \mathrm{order}(q^\eta)$, hence:
\begin{equation}
\label{equation:lemma_choice_parameters_omega}
\omega = \frac{e}{\mathrm{gcd}(e, \eta)},
\end{equation}
thus $e = \mathrm{gcd}(e, \eta)\omega \mid \eta\omega$.
As a consequence, we have a well-defined \emph{surjective} map:
\[
\begin{array}{c|rcl}
\overline{\phi} : &\mathbb{Z}/\omega\mathbb{Z} & \longrightarrow & \eta I
\\
&a & \longmapsto & \eta a
\end{array}.
\]
We have, using \eqref{equation:lemma_choice_parameters_omega}:
\[
\eta I =  (\eta\mathbb{Z} + e\mathbb{Z})/e\mathbb{Z} \simeq \mathbb{Z}/{\textstyle \frac{e}{\mathrm{gcd}(e, \eta)}}\mathbb{Z} = \mathbb{Z}/\omega\mathbb{Z},
\]
thus, by a counting argument we get that the map $\overline{\phi}$ is  bijective. This concludes the proof.
\end{proof}

The second aim of this subsection is achieved thanks to the following proposition.

\begin{proposition}
\label{proposition:removing_repetitions-_Lambda}
A weight $\tuple{\Lambda} \in \mathbb{N}^{(I)}$ of level $\omega d$ comes  from a weight $\tuple{\varLambda} \in \mathbb{N}^{(I)}$ of level $d$ through \eqref{equation:removing_repetitions-definition_Lambda_w/o_j} if and only if for all $i \in I$,
\[
\Lambda_i = \Lambda_{i + \eta},
\]
that is, if and only if the weight $\tuple{\Lambda}$ factors to a weight $\tuple{\Lambda} \in \mathbb{N}^{(I/\eta I)}$ of level $d$.
\end{proposition}

\begin{proof}
First, by applying Lemma~\ref{lemma:remove_repetitions} to  \eqref{equation:removing_repetitions-definition_Lambda_w/o_j}, we obtain the equivalent equality:
\begin{equation}
\label{equation:removing_repetitions-Lambda_widetildeLambda_w/o_j} 
\Lambda_i = \sum_{\substack{i' \in I \\ i' - i \in \eta I}} \varLambda_{i'} = \sum_{i' \in i + \eta I} \varLambda_{i'},
\end{equation}
for all $i \in I$.
The necessary condition is hence straightforward. We now suppose that $\tuple{\Lambda} \in \mathbb{N}^{(I)}$ factors to a weight  $\tuple{\Lambda} \in \mathbb{N}^{(I/\eta I)}$ of level $d$. 
 For any $\gamma \in I/\eta I$, we choose any $\omega$ non-negative integers $\varLambda_i$ for $i \in \gamma$ such that $\sum_{i \in \gamma} \varLambda_i = \Lambda_\gamma$. We conclude that \eqref{equation:removing_repetitions-Lambda_widetildeLambda_w/o_j} and thus \eqref{equation:removing_repetitions-definition_Lambda} hold since $\Lambda_i = \Lambda_\gamma$ if $i \in \gamma$.
\end{proof}

\begin{definition}
\label{definition:removing_repetitions-Hn_Hpn}
We write $\H_n^{\tuple{\Lambda}}(q, \zeta) \coloneqq \H_n^{\tuple{\varLambda}}(q, \zeta)$ and $\H_{p, n}^{\tuple{\Lambda}}(q) \coloneqq \H_{p, n}^{\tuple{\varLambda}}(q)$ if $\tuple{\Lambda} \in \mathbb{N}^{(I)}$ of level $\omega d$ and $\tuple{\varLambda} \in \mathbb{N}^{(I)}$ of level $d$ are as in Proposition~\ref{proposition:removing_repetitions-_Lambda}. In particular, the cyclotomic relation \eqref{relation:hecke_cyclo_S} in $\H_n^{\tuple{\Lambda}}(q, \zeta)$ is:
\[
\prod_{i \in I} \prod_{j \in J'} {(S - \zeta^j q^i)}^{\Lambda_i} = 0.
\]
\end{definition}

\section{Cyclotomic quiver Hecke algebras}
\label{section:CQHA}

Let $K$ be a set and $\Gamma$ be a loop-free quiver without multiple (directed) edges, with vertex set $K$. For $k, k' \in K$:
\begin{itemize}
\item we write $k \nrelbar k'$ if neither $(k, k')$ nor $(k', k)$ is an edge of $\Gamma$;
\item we write $k \to k'$ if $(k, k')$ is an edge of $\Gamma$ and $(k', k)$ is not;
\item we write $k \leftarrow k'$ if $k' \to k$;
\item we write $k \leftrightarrows k'$ if both $(k, k')$ and $(k', k)$ are edges of $\Gamma$.
\end{itemize}
If the next section we will define the (cyclotomic) quiver Hecke algebra associated with $\Gamma$. Furthermore, given an automorphism of this algebra built from an automorphism of the quiver, we will give a presentation of the fixed point subalgebra. Note that what follows has roughly the same outline as \cite{Bo, BoMa}.

\subsection{Definition}

Let $n \in \mathbb{N}^*$ and $\alpha = (\alpha_k)_{k \in K}$ be a (finitely supported) tuple of non-negative integers, whose sum is equal to $n$. We say that $\alpha$ is a \emph{$K$-composition of $n$} and we write $\alpha \comp_{K} n$.
Let $K^\alpha$ be the subset of $K^n$ consisting of the elements of $K^n$ which have, for any $k \in K$, exactly $\alpha_k$ components equal to $k$.
The sets $K^{\alpha}$ are the $\mathfrak{S}_n$-orbits of $K^n$, in particular they are finite.

Following \cite{KhLau1, KhLau2} and \cite{Rou}, we denote by $\sqH_\alpha(\Gamma)$ the \emph{quiver Hecke algebra  at $\alpha$ associated with $\Gamma$}. It is the unitary associative $F$-algebra  with generating set
\[
\{e(\tuple{k})\}_{\tuple{k} \in K^\alpha} \cup \{y_1, \dots, y_n\} \cup \{\psi_1, \dots, \psi_{n-1}\}
\]
together with the relations
\begin{subequations}
\label{relations:affine_quiver}
\begin{align}
\label{relation:quiver_sum_e(k)}
\sum_{\tuple{k} \in K^\alpha} e(\tuple{k}) &= 1,
\\
\label{relation:quiver_e(k)e(k')}
e(\tuple{k})e(\tuple{k}') &= \delta_{\tuple{k}, \tuple{k}'} e(\tuple{k}) ,
\\
\label{relation:quiver_y_ae(k)}
y_a e(\tuple{k}) &= e(\tuple{k}) y_a,
\\
\label{relation:quiver_psiae(k)}
\psi_a e(\tuple{k}) &= e(s_a \cdot \tuple{k}) \psi_a,
\\
\label{relation:quiver_ya_yb}
y_a y_b &= y_b y_a,
\\
\label{relation:quiver_psia_yb}
\psi_a y_b &= y_b \psi_a \quad \text{if } b \neq a, a+1,
\\
\label{relation:quiver_psia_psib}
\psi_a \psi_b &= \psi_b \psi_a \quad \text{if } |a-b| > 1,
\\
\label{relation:quiver_psia_ya+1}
\psi_a y_{a+1} e(\tuple{k}) &= \begin{cases}
(y_a \psi_a + 1)e(\tuple{k}) & \text{if } k_a = k_{a+1}, \\
y_a \psi_a e(\tuple{k}) & \text{if } k_a \neq k_{a+1},
\end{cases}
\\
\label{relation:quiver_ya+1_psia}
y_{a+1} \psi_a e(\tuple{k}) &= \begin{cases}
(\psi_a y_a + 1)e(\tuple{k}) & \text{if } k_a = k_{a+1}, \\
\psi_a y_a e(\tuple{k}) & \text{if } k_a \neq k_{a+1},
\end{cases}
\\
\displaybreak[0]
\label{relation:quiver_psia^2}
\psi_a^2 e(\tuple{k}) &= \begin{cases}
0 & \text{if } k_a = k_{a+1}, \\
e(\tuple{k}) & \text{if } k_a \nrelbar k_{a+1}, \\
(y_{a+1} - y_a)e(\tuple{k}) & \text{if } k_a \to k_{a+1}, \\
(y_a - y_{a+1})e(\tuple{k}) & \text{if } k_a \leftarrow k_{a+1}, \\
(y_{a+1} - y_a)(y_a - y_{a+1})e(\tuple{k}) & \text{if } k_a \leftrightarrows k_{a+1},
\end{cases}
\\
\label{relation:quiver_tresse}
\psi_{a+1}\psi_a \psi_{a+1}e(\tuple{k}) &= \begin{cases}
(\psi_a \psi_{a+1}\psi_a -1)e(\tuple{k}) & \text{if } k_{a+2} = k_a \to k_{a+1}, \\
(\psi_a \psi_{a+1}\psi_a +1)e(\tuple{k}) & \text{if } k_{a+2} = k_a \leftarrow k_{a+1}, \\
(\psi_a \psi_{a+1} \psi_a + 2y_{a+1} - y_a - y_{a+2})e(\tuple{k}) & \text{if } k_{a+2} = k_a \leftrightarrows k_{a+1}, \\
\psi_a \psi_{a+1} \psi_a e(\tuple{k}) & \text{otherwise.}
\end{cases}
\end{align}
\end{subequations}

The following proposition is easy to check (see \cite{KhLau1,KhLau2} or \cite{Rou}).

\begin{proposition}
\label{proposition:gradation_qH}
There is a unique $\mathbb{Z}$-grading on 
 $\sqH_\alpha(\Gamma)$ such that:
\begin{itemize}
\item for all $\tuple{k} \in K^\alpha, \deg e(\tuple{k}) \coloneqq 0$;
\item for all $\tuple{k} \in K^\alpha$ and $a \in \{1, \dots, n\}, \deg y_a e(\tuple{k}) \coloneqq 2$;
\item for all $\tuple{k} \in K^\alpha$ and $a \in \{1, \dots, n-1\}$:
\[
\deg \psi_a e(\tuple{k}) \coloneqq \begin{cases}
-2 & \text{if } k_a = k_{a+1},
\\
0 & \text{if } k_a \nrelbar k_{a+1},
\\
1 & \text{if } k_a \to k_{a+1} \text{ or } k_a \leftarrow k_{a+1},
\\
2 & \text{if } k_a \leftrightarrows k_{a+1}.
\end{cases}
\]
\end{itemize}
\end{proposition}

If $K$ is finite, we can define the quiver Hecke algebra $\sqH_n(\Gamma)$ as above, with generators
\[
\{e(\tuple{k})\}_{\tuple{k} \in K^n} \cup \{y_1, \dots, y_n\} \cup \{\psi_1, \dots, \psi_{n-1}\}
\]
and the relations
\eqref{relations:affine_quiver}, where the sum over $K^\alpha$ in \eqref{relation:quiver_sum_e(k)} is now over $K^n$. In that case, define for $\alpha \comp_{K} n$:
\begin{equation}
\label{equation:definition_e(alpha)_QH}
e(\alpha) \coloneqq \sum_{\tuple{k} \in K^\alpha} e(\tuple{k}) \in \sqH_n(\Gamma).
\end{equation}
The element $e(\alpha)$ is central and $\sqH_n(\Gamma)e(\alpha) \simeq \sqH_\alpha(\Gamma)$. We have:
\[
\sqH_n(\Gamma) \simeq \bigoplus_{\alpha \comp_{K} n} \sqH_\alpha(\Gamma),
\]
and this isomorphism can be considered as a definition of $\sqH_n(\Gamma)$ when  $K$ is infinite (note that in that case, the algebra $\sqH_n(\Gamma)$ is not unitary).

Now let $\tuple{\Lambda} = (\Lambda_k)_{k \in K} \in \mathbb{N}^{(K)}$ be a finitely-supported tuple of non-negative integers. We  define a particular case of cyclotomic quotient of $\sqH_\alpha(\Gamma)$ (see \cite{KanKa} for the general case).

\begin{definition}
\label{definition:cyclo_qH}
The \emph{cyclotomic} quiver Hecke algebra $\sqH_\alpha^{\tuple{\Lambda}}(\Gamma)$ is the quotient of the quiver Hecke algebra $\sqH_\alpha(\Gamma)$ by the two-sided ideal $\mathcal{I}_\alpha^{\tuple{\Lambda}}$ generated by the following relations:
\begin{equation}
\label{relation:quiver_cyclo}
y_1^{\Lambda_{k_1}} e(\tuple{k}) = 0, \qquad \text{for all }  \tuple{k} \in K^\alpha.
\end{equation}
\end{definition}
Note that the grading on $\sqH_\alpha(\Gamma)$ gives a grading on $\sqH_\alpha^{\tuple{\Lambda}}(\Gamma)$.
If $K$ is finite, the cyclotomic quotient $\sqH_n^{\tuple{\Lambda}}(\Gamma)$ is obtained by quotienting $\sqH_n(\Gamma)$ by the relations:
\[
y_1^{\Lambda_{k_1}} e(\tuple{k}) = 0, \qquad \text{for all } \tuple{k} \in K^n.
\]
Moreover, we have $e(\alpha)\sqH_n^{\tuple{\Lambda}}(\Gamma) \simeq \sqH_\alpha^{\tuple{\Lambda}}(\Gamma)$ and:
\[
\sqH_n^{\tuple{\Lambda}}(\Gamma) \simeq \bigoplus_{\alpha \comp_{K} n} \sqH_\alpha^{\tuple{\Lambda}}(\Gamma),
\]
and this can be considered as a definition of $\sqH_n^{\tuple{\Lambda}}(\Gamma)$ if $K$ is infinite.

\subsection{Properties of the underlying vector spaces}

For each $w \in \mathfrak{S}_n$, we choose a reduced expression $w = s_{a_1} \cdots s_{a_r}$ and we set:
\begin{equation}
\label{equation:definition_psi_w}
\psi_w \coloneqq \psi_{a_1} \cdots \psi_{a_r} \in \sqH_n(\Gamma).
\end{equation}
Let $\alpha \comp_{K} n$. We have the following theorem (\cite[Theorem 3.7]{Rou}, \cite[Theorem 2.5]{KhLau1}).

\begin{theoreme}
\label{theorem:base_quiver}
The family $\mathcal{B}_\alpha \coloneqq \{\psi_w y_1^{m_1} \cdots y_n^{m_n} e(\tuple{k}) : w \in \mathfrak{S}_n, m_a \in \mathbb{N}, \tuple{k} \in K^\alpha\}$ is a basis of the $F$-vector space $\sqH_\alpha(\Gamma)$.
\end{theoreme} 

Let $\tuple{\Lambda} \in \mathbb{N}^{(K)}$ be a weight. It is not obvious that we can deduce a basis of the cyclotomic quiver Hecke algebra from the basis of Theorem~\ref{theorem:base_quiver}. With this in mind, let us give the following lemma.

\begin{lemme}[\protect{\cite[Lemma 2.1]{BrKl}}]
\label{lemma:ya_nilpotent}
The elements $y_a \in \sqH_\alpha^{\tuple{\Lambda}}(\Gamma)$ are nilpotent for any $a \in \{1, \dots, n\}$.
\end{lemme}

\begin{remarque}
The proof in \cite{BrKl} is given for the quiver $\Gamma_e$. However, we can see immediately that their proof is valid for the quivers we use here.
\end{remarque}

We obtain the following theorem (see also \cite[\textsection 4.1]{KanKa}).

\begin{theoreme}
\label{theorem:generating_family_quiver_cyclo}
The family:
\[
\mathcal{B}_\alpha^{\tuple{\Lambda}} \coloneqq \big\lbrace \psi_w y_1^{m_1} \cdots y_n^{m_n} e(\tuple{k}) : w \in \mathfrak{S}_n, m_a \in \mathbb{N}, \tuple{k} \in K^\alpha\big\rbrace,
\]
is finite and spans $\sqH_\alpha^{\tuple{\Lambda}}(\Gamma)$ over $F$. In particular, the $F$-vector space $\sqH_\alpha^{\tuple{\Lambda}}(\Gamma)$ is finite-dimensional.
\end{theoreme}

\subsection{Fixed point subalgebra}
\label{subsection:fixed_subalgebra}

Let $\sigma$ be an automorphism of $\Gamma$, that is:
\begin{itemize}
\item the map $\sigma : K \to K$ is a bijection;
\item if $(k, k') \in K^2$ is an edge of $\Gamma$ then $(\sigma(k), \sigma(k')) \in K^2$ is also an edge of $\Gamma$.
\end{itemize}

We assume that for any $p_1 \in \{1, \dots, p-1\}$ and for any $k \in K$ we have $\sigma^{p_1}(k) \neq k = \sigma^p(k)$, in particular, we have $\sigma^p = \mathrm{id}_K$.

\begin{remarque}
Everything in this subsection \textsection\ref{subsection:fixed_subalgebra} remains true if we only assume that $\sigma^p = \mathrm{id}_K$: we just have to modify the sentence containing \eqref{equation:gamma_p'} and equalities \eqref{equation:e(gamma)_sum_e(sigma^m(k))}, \eqref{equation:cyclotomic_case-to_change_p'}. Note that the automorphism involved in the proof of the main theorem (the automorphism which is defined in \eqref{equation:intertwining-definition_sigma}) satisfies the above stronger condition.
\end{remarque}

Since $\sigma^p = \mathrm{id}_K$, we have, for $k, k' \in K$:
\[
(k, k') \in K^2 \text{ is an edge of } \Gamma \iff (\sigma(k), \sigma(k')) \in K^2 \text{ is an edge of } \Gamma,
\]
thus we deduce the following lemma.

\begin{lemme}
\label{lemma:->_compatible}
Let $k, k' \in K$. We have:
\[
k \to k' \iff \sigma(k) \to \sigma(k),
\]
and there are similar equivalences for $\leftarrow, \leftrightarrows, =$ and $\neq$.
\end{lemme}

The map $\sigma$ naturally induces a map $\sigma : K^n \to K^n$, defined by $\sigma(\tuple{k}) \coloneqq (\sigma(k_1), \dots, \sigma(k_n))$. 
For $\alpha \comp_{K} n$, the following lemma explains how $\sigma : K^n \to K^n$ restricts to $K^\alpha$ (compare to \cite[after Lemma 5.3.2]{Bo}).

\begin{lemme}
\label{lemma:fixed_subalgebra-sigma_cdot_alpha}
For $\alpha \comp_{K} n$, the map $\sigma : K^n \to K^n$ maps $K^\alpha$ onto $K^{\sigma \cdot \alpha}$, where $\sigma \cdot \alpha$ is the $K$-composition of $n$ given by:
\[
(\sigma \cdot \alpha)_k \coloneqq \alpha_{\sigma^{-1}(k)}, \qquad \text{for all } k \in K.
\]
\end{lemme}

\begin{proof}
Let $\tuple{k} \in K^n$. We have:
\begin{align*}
&
\text{for all } k \in K, \tuple{k} \text{ has } \alpha_k \text{ components equal to } k
\\
\iff&
\text{for all }  k \in K, \sigma(\tuple{k}) \text{ has } \alpha_k \text{ components equal to } \sigma(k)
\\
\iff&
\text{for all }  k \in K, \sigma(\tuple{k}) \text{ has } \alpha_{\sigma^{-1}(k)} \text{ components equal to } k
\\
\iff&
\text{for all }  k \in K, \sigma(\tuple{k}) \text{ has } (\sigma \cdot \alpha)_k \text{ components equal to } k.
\end{align*}
We conclude that $\tuple{k} \in K^\alpha \iff \sigma(\tuple{k}) \in K^{\sigma \cdot \alpha}$.
%
\end{proof}

We  can now explain how $\sigma$ induces an isomorphism between (cyclotomic) quiver Hecke algebras. We will also give a presentation for the fixed point subalgebras.

\subsubsection{Affine case}
\label{subsubsection:fixed_affine_subalgebra}

In the affine case, we will be able to give a basis for the subalgebra of the fixed points of $\sigma$. As an easy consequence, we will give a presentation of this subalgebra.

\begin{theoreme}
\label{theorem:definition_sigma_quiver}
Let $\alpha \comp_{K} n$. There is a well-defined algebra homomorphism $\sigma : \sqH_\alpha(\Gamma) \to \sqH_{\sigma \cdot \alpha}(\Gamma)$ given by:
\begin{align*}
\sigma(e(\tuple{k})) &\coloneqq e(\sigma(\tuple{k})), &&\text{for all } \tuple{k} \in K^\alpha,
\\
\sigma(y_a) &\coloneqq y_a, &&\text{for all } a \in \{1, \dots, n\},
\\
\sigma(\psi_a) &\coloneqq \psi_a, &&\text{for all } a \in \{1, \dots, n-1\}.
\end{align*}
\end{theoreme}

\begin{proof}
We check the different relations
\eqref{relations:affine_quiver}, thanks to Lemma~\ref{lemma:->_compatible} and the following fact:
\[
\sigma(\tuple{k})_a = \sigma(k_a),
\]
for all $a \in \{1, \dots, n\}$ and $\tuple{k} \in K^n$.
Note that to prove \eqref{relation:quiver_sum_e(k)} we use the additional fact that $\sigma : K^\alpha \to K^{\sigma \cdot \alpha}$ is a bijection.
\end{proof}

\begin{remarque}
\label{remark:sigma_qH_graded}
By Lemma~\ref{lemma:->_compatible}, the homomorphism $\sigma : \sqH_\alpha(\Gamma) \to \sqH_{\sigma \cdot \alpha}(\Gamma)$ is homogeneous with respect to the grading given in Proposition~\ref{proposition:gradation_qH}.
\end{remarque}

As in Section~\ref{section:AK}, we want to study the fixed points of $\sigma$. To that extent, we first need to find an algebra which is stable under $\sigma$.
Let $[\alpha]$ be the orbit of $\alpha$ under the action of $\langle\sigma\rangle$. Note that since $\sigma^p = \mathrm{id}_{K}$, the cardinality of $[\alpha]$ is at most $p$. For $\alpha \comp_{K} n$ we define the following finite subset of $K^n$:
\begin{equation}
\label{equation:affine_case-K[alpha]}
K^{[\alpha]} \coloneqq \bigsqcup_{\beta \in [\alpha]} K^\beta,
\end{equation}
and similarly we define the following unitary algebra:
\begin{equation}
\label{equation:definition_qH[alpha]}
\sqH_{[\alpha]}(\Gamma) \coloneqq \bigoplus_{\beta \in [\alpha]} \sqH_\beta(\Gamma).
\end{equation}
 We obtain an \emph{automorphism} $\sigma : \sqH_{[\alpha]}(\Gamma) \to \sqH_{[\alpha]}(\Gamma)$.

\begin{remarque}
We have $\sqH_n(\Gamma) \simeq \oplus_{[\alpha]} \sqH_{[\alpha]}(\Gamma)$, in particular,  for $\tuple{k} \in K^n$ the idempotent $e(\tuple{k})$ of $\sqH_n(\Gamma)$ belongs to $\sqH_{[\alpha]}(\Gamma)$ if and only if $\tuple{k} \in K^{[\alpha]}$.
\end{remarque}

We consider the equivalence relation $\sim$ on $K$ generated by:
\begin{equation}
\label{equation:definition_tilde}
k \sim \sigma(k), \qquad \text{for all }  k \in K.
\end{equation}
We extend it to $K^{[\alpha]}$ by:
\begin{equation}
\label{equation:definition_tilde_tuple}
\tuple{k} \sim \sigma(\tuple{k}), \qquad \text{for all } \tuple{k} \in K^{[\alpha]}.
\end{equation}

\begin{definition}
\label{definition:affine_case-Kalphasigma}
We write $K^{[\alpha]}_\sigma$ for the quotient set $K^{[\alpha]} / {\sim}$.
\end{definition}

 In particular, each element $\gamma \in K^{[\alpha]}_\sigma$ has cardinality $p$ and is of the form:
\begin{equation}
\label{equation:gamma_p'}
\gamma = \{\tuple{k}, \sigma(\tuple{k}), \dots, \sigma^{p - 1}(\tuple{k})\},
\end{equation}
with $\tuple{k} \in K^{[\alpha]}$.

\begin{definition}
For $\gamma \in K^{[\alpha]}_\sigma$, we define:
\[
e(\gamma) \coloneqq \sum_{\substack{\tuple{k} \in K^{[\alpha]} \\ \tuple{k} \in \gamma}} e(\tuple{k}).
\]
\end{definition}

These elements $e(\gamma)$ have the property of being fixed by $\sigma$. Note that for any $\tuple{k} \in \gamma$, by \eqref{equation:gamma_p'} we have:
\begin{equation}
\label{equation:e(gamma)_sum_e(sigma^m(k))}
e(\gamma) = \sum_{m = 0}^{p-1} e(\sigma^m(\tuple{k})).
\end{equation}

We now give the analogue of Proposition~\ref{proposition:basis_fixed_point_Hecke}, by describing all the fixed points of $\sigma$.

\begin{theoreme}
\label{theorem:quiver_fixed_point}
The following family:
\[
\mathcal{B}_{[\alpha]}^\sigma \coloneqq \big\lbrace\psi_w y_1^{m_1} \cdots y_n^{m_n} e(\gamma) : w \in \mathfrak{S}_n, m_a \in \mathbb{N}, \gamma \in K^{[\alpha]}_\sigma\big\rbrace,
\]
is an $F$-basis of $\sqH_{[\alpha]}(\Gamma)^\sigma$, the vector space of $\sigma$-fixed points of $\sqH_{[\alpha]}(\Gamma)$.
\end{theoreme}

\begin{proof}
First, by Theorem~\ref{theorem:base_quiver} we know that $\mathcal{B}_{[\alpha]} \coloneqq \sqcup_{\beta \in [\alpha]} \mathcal{B}_\beta$ is a linear basis of $\sqH_{[\alpha]}(\Gamma)$.
Hence, the family $\mathcal{B}_{[\alpha]}^\sigma$ is linearly independent. Moreover, each element of $\mathcal{B}_{[\alpha]}^\sigma$ is fixed by $\sigma$.
Now let $h \in \sqH_{[\alpha]}(\Gamma)$ be fixed by $\sigma$: we want to prove that $h$ lies in $\mathrm{span}_F(\mathcal{B}_{[\alpha]}^\sigma)$. Using Theorem~\ref{theorem:base_quiver}, we can write, with $\tuple{m} = (m_a)_a$:
\[
h = \sum_{w \in \mathfrak{S}_n} \sum_{\tuple{m} \in \mathbb{N}^n} \sum_{\tuple{k} \in K^{[\alpha]}} h_{w, \tuple{m}, \tuple{k}} \psi_w y_1^{m_1} \cdots y_n^{m_n} e(\tuple{k}),
\]
where $h_{w, \tuple{m}, \tuple{k}} \in F$. We have:
\[
h = \sigma(h) =  \sum_{w \in \mathfrak{S}_n} \sum_{\tuple{m} \in \mathbb{N}^n} \sum_{\tuple{k} \in K^{[\alpha]}} h_{w, \tuple{m}, \tuple{k}} \psi_w y_1^{m_1} \cdots y_n^{m_n} e(\sigma(\tuple{k})),
\]
and thus, since $\mathcal{B}_{[\alpha]}$  is linearly independent:
\[
h_{w, \tuple{m}, \tuple{k}} = h_{w, \tuple{m}, \sigma(\tuple{k})},
\]
for all $w \in \mathfrak{S}_n, \tuple{m} \in \mathbb{N}^n$ and $\tuple{k} \in K^{[\alpha]}$. In particular, for each $\gamma \in K^{[\alpha]}_\sigma$ there is a well-defined scalar $h_{w, \tuple{m}, \gamma}$. We obtain:
\begin{align*}
h 
&= \sum_{w \in \mathfrak{S}_n} \sum_{\tuple{m} \in \mathbb{N}^n} \sum_{\gamma \in K^{[\alpha]}_\sigma} h_{w, \tuple{m}, \gamma} \psi_w y_1^{m_1} \cdots y_n^{m_n} \sum_{\tuple{k} \in \gamma} e(\tuple{k})
\\
&= \sum_{w \in \mathfrak{S}_n} \sum_{\tuple{m} \in \mathbb{N}^n} \sum_{\gamma \in K^{[\alpha]}_\sigma} h_{w, \tuple{m}, \gamma} \psi_w y_1^{m_1} \cdots y_n^{m_n} e(\gamma),
\end{align*}
thus $h$ lies in $\mathrm{span}_F(\mathcal{B}_{[\alpha]}^\sigma)$.
\end{proof}

Theorem~\ref{theorem:quiver_fixed_point} allows us to give a presentation of the algebra $\sqH_{[\alpha]}(\Gamma)^\sigma$. First, let us note that for any $a, b, c \in \{1, \dots, n\}$ with $c < n$ and $\gamma \in K^{[\alpha]}_\sigma$,   the expressions:
\begin{gather*}
s_c \cdot \gamma, \quad
\gamma_a = \gamma_b, \quad \gamma_a \neq \gamma_b,
\\
\gamma_a \to \gamma_b, \quad \gamma_a \leftarrow \gamma_b, \quad \gamma_a \leftrightarrows \gamma_b,
\end{gather*}
are well-defined,  thanks to Lemma~\ref{lemma:->_compatible}.

\begin{remarque}
\label{remark:gammaa_gamma'a}
In contrast, if $\gamma'$ is another element of $K^{[\alpha]}_\sigma$ then the expression $\gamma_a = \gamma'_a$ (for instance) is not well-defined. In particular, recalling \eqref{equation:definition_tilde} and \eqref{equation:definition_tilde_tuple}, instead of subsets of $K^n_\sigma = K^n / {\sim}$ we may want to consider subsets of $(K / {\sim})^n$. This set $(K/{\sim})^n$ has much worse properties. For instance, denoting by $\tuple{\kappa} \in (K /{\sim})^n$ the image of  $\tuple{k} \in K^n$ (or $\gamma \in K^n_\sigma$), if $\kappa_a = \kappa_{a+1}$ then we may have $k_a \neq k_{a+1}$ (or $\gamma_a \neq \gamma_{a+1}$).
\end{remarque}

\begin{corollaire}
\label{corollary:presentation_quiver_fixed}
The algebra $\sqH_{[\alpha]}(\Gamma)^\sigma$ has the following presentation. The generating set is:
\begin{equation}
\label{equation:generators_quiver_fixed}
\{e(\gamma)\}_{\gamma \in K^{[\alpha]}_\sigma} \cup \{y_1, \dots, y_n\} \cup \{\psi_1, \dots, \psi_{n-1}\},
\end{equation}
and the relations are:
\begin{subequations}
\label{relations:affine_quiver_fixed}
\begin{align}
\label{relation:quiver_sum_e(gamma)_fixed}
\sum_{\gamma \in K^{[\alpha]}_\sigma} e(\gamma) &= 1,
\\
\label{relation:quiver_e(gamma)e(gamma')_fixed}
e(\gamma)e(\gamma') &= \delta_{\gamma, \gamma'} e(\gamma),
\\
\label{relation:quiver_y_ae(gamma)_fixed}
y_a e(\gamma) &= e(\gamma) y_a,
\\
\label{relation:quiver_psiae(gamma)_fixed}
\psi_a e(\gamma) &= e(s_a \cdot \gamma) \psi_a,
\\
\label{relation:quiver_ya_yb_fixed}
y_a y_b &= y_b y_a,
\\
\label{relation:quiver_psia_yb_fixed}
\psi_a y_b &= y_b \psi_a \quad \text{if } b \neq a, a+1,
\\
\label{relation:quiver_psia_psib_fixed}
\psi_a \psi_b &= \psi_b \psi_a \quad \text{if } |a-b| > 1,
\\
\label{relation:quiver_psia_ya+1_fixed}
\psi_a y_{a+1} e(\gamma) &= \begin{cases}
(y_a \psi_a + 1)e(\gamma) & \text{if } \gamma_a = \gamma_{a+1}, \\
y_a \psi_a e(\gamma) & \text{if } \gamma_a \neq \gamma_{a+1},
\end{cases}
\\
\label{relation:quiver_ya+1_psia_fixed}
y_{a+1} \psi_a e(\gamma) &= \begin{cases}
(\psi_a y_a + 1)e(\gamma) & \text{if } \gamma_a = \gamma_{a+1}, \\
\psi_a y_a e(\gamma) & \text{if } \gamma_a \neq \gamma_{a+1},
\end{cases}
\\
\label{relation:quiver_psia^2_fixed}
\psi_a^2 e(\gamma) &= \begin{cases}
0 & \text{if } \gamma_a = \gamma_{a+1}, \\
e(\gamma) & \text{if } \gamma_a \nrelbar \gamma_{a+1}, \\
(y_{a+1} - y_a)e(\gamma) & \text{if } \gamma_a \to \gamma_{a+1}, \\
(y_a - y_{a+1})e(\gamma) & \text{if } \gamma_a \leftarrow \gamma_{a+1}, \\
(y_{a+1} - y_a)(y_a - y_{a+1})e(\gamma) & \text{if } \gamma_a \leftrightarrows \gamma_{a+1},
\end{cases}
\\
\label{relation:quiver_tresse_fixed}
\psi_{a+1}\psi_a \psi_{a+1}e(\gamma) &= \begin{cases}
(\psi_a \psi_{a+1}\psi_a -1)e(\gamma) & \text{if } \gamma_{a+2} = \gamma_a \to \gamma_{a+1}, \\
(\psi_a \psi_{a+1}\psi_a +1)e(\gamma) & \text{if } \gamma_{a+2} = \gamma_a \leftarrow \gamma_{a+1}, \\
(\psi_a \psi_{a+1} \psi_a + 2y_{a+1} - y_a - y_{a+2})e(\gamma) & \text{if } \gamma_{a+2} = \gamma_a \leftrightarrows \gamma_{a+1}, \\
\psi_a \psi_{a+1} \psi_a e(\gamma) & \text{otherwise.}
\end{cases}
\end{align}
\end{subequations}
\end{corollaire}

\begin{proof}
Let us temporary write $e(\gamma)^\sigma, y_a^\sigma$ and $\psi_a^\sigma$ for the generators of Corollary~\ref{corollary:presentation_quiver_fixed}, and write $\sqH^\sigma$ for the algebra which admits this presentation. Given the defining relations
\eqref{relations:affine_quiver} of $\sqH_{[\alpha]}(\Gamma)$ and Lemma~\ref{lemma:->_compatible}, there is a well-defined  algebra homomorphism $f : \sqH^\sigma \to \sqH_{[\alpha]}(\Gamma)$ given by:
\begin{align*}
f(e(\gamma)^\sigma) &\coloneqq e(\gamma), &&\text{for all } \gamma \in K^{[\alpha]}_\sigma,
\\
f(y_a^\sigma) &\coloneqq y_a, &&\text{for all } a \in \{1, \dots, n\},
\\
f(\psi_a^\sigma) &\coloneqq \psi_a, &&\text{for all } a \in \{1, \dots, n-1\}.
\end{align*}

We can notice that the family:
\[
\mathcal{B}^\sigma \coloneqq \big\lbrace\psi_w^\sigma {(y_1^\sigma)}^{m_1} \cdots {(y_n^\sigma)}^{m_n} e(\gamma)^\sigma : w \in \mathfrak{S}_n, m_a \in \mathbb{N}, \gamma \in K^{[\alpha]}_\sigma\big\rbrace,
\]
spans $\sqH^\sigma$ over $F$, where the elements $\psi_w^\sigma$ are defined as in \eqref{equation:definition_psi_w}, with the same reduced expressions. We recall from Theorem~\ref{theorem:quiver_fixed_point}  that the family
\[
\mathcal{B}_{[\alpha]}^\sigma = \big\lbrace\psi_w y_1^{m_1} \cdots y_n^{m_n} e(\gamma) : w \in \mathfrak{S_n}, m_a \in \mathbb{N}, \gamma \in K^{[\alpha]}_\sigma\big\rbrace
\]
is an $F$-basis of $\sqH_{[\alpha]}^\sigma(\Gamma)$. Noticing that the algebra homomorphism $f$ maps $\mathcal{B}^\sigma$ onto $\mathcal{B}_{[\alpha]}^\sigma$, we deduce that:
\begin{itemize}
\item the family $\mathcal{B}^\sigma$ is linearly independent;
\item the map $f$ surjects onto $\sqH_{[\alpha]}(\Gamma)^\sigma$.
\end{itemize}
Finally, the family $\mathcal{B}^\sigma$ is a basis of $\sqH^\sigma$. In particular, the homomorphism $f$ sends a basis to a basis hence $f$ is an isomorphism.
\end{proof}

The reader may have noticed the similarity between the relations~\eqref{relations:affine_quiver_fixed} defining $\sqH_{[\alpha]}(\Gamma)^\sigma$ and the relations~\eqref{relations:affine_quiver} defining $\sqH_\alpha(\Gamma)$. However, now the indexing set for the idempotents is generally not an $\mathfrak{S}_n$-stable subset of $\mathcal{I}^n$ for $\mathcal{I}$ an indexing set.

\begin{remarque}
\label{remark:gradation_qH_fixed}
Since $\sigma : \sqH_{[\alpha]}(\Gamma) \to \sqH_{[\alpha]}(\Gamma)$ is homogeneous (cf. Remark~\ref{remark:sigma_qH_graded}), the subalgebra $\sqH_{[\alpha]}(\Gamma)^\sigma$ is a graded subalgebra of $\sqH_{[\alpha]}(\Gamma)$. More precisely, as in Proposition~\ref{proposition:gradation_qH} there is a unique $\mathbb{Z}$-grading on $\sqH_{[\alpha]}(\Gamma)^\sigma$ such that $e(\gamma)$ is of degree $0$, the element $y_a$ is of degree $2$ and $\psi_a e(\gamma)$ is of degree $-c_{\gamma_a, \gamma_{a+1}}$ where:
\[
c_{\gamma_a, \gamma_{a+1}} \coloneqq \begin{cases}
2 & \text{if } \gamma_a = \gamma_{a+1},
\\
0 & \text{if } \gamma_a \nrelbar \gamma_{a+1},
\\
-1 & \text{if } \gamma_a \to \gamma_{a+1} \text{ or } \gamma_{a+1} \to \gamma_a,
\\
-2 & \text{if } \gamma_a \leftrightarrows \gamma_{a+1}.
\end{cases}
\]
\end{remarque}

\subsubsection{Cyclotomic case}
\label{subsubsection:fixed_cyclotomic_subalgebra}

Recall the Definition~\ref{definition:cyclo_qH} of a cyclotomic quiver Hecke algebra.
For $\alpha \comp_{K} n$, we want the algebra homomorphism $\sigma : \sqH_\alpha(\Gamma) \to \sqH_{\sigma \cdot \alpha}(\Gamma)$ to factor through cyclotomic quotients. Contrary to the affine case, it will be more difficult to get a presentation for the fixed point subalgebra of the cyclotomic quiver Hecke algebra (recall that we do not have an analogue of Theorem~\ref{theorem:base_quiver}). In particular, the whole proof relies on the map $\mu$ which will be introduced in \eqref{equation:cyclotomic-case_mu}.

Let $\tuple{\Lambda} \in \mathbb{N}^{(K)}$ be a weight. As for $K$-compositions, we define the weight $\sigma \cdot \tuple{\Lambda} \in \mathbb{N}^{(K)}$ by:
\[
(\sigma \cdot \tuple{\Lambda})_k \coloneqq \Lambda_{\sigma^{-1}(k)}, \qquad \text{for all } k \in K.
\]

\begin{lemme}
\label{lemma:sigma_Ilambda}
We have $\sigma(\mathcal{I}_\alpha^{\tuple{\Lambda}}) = \mathcal{I}_{\sigma \cdot \alpha}^{\sigma \cdot \tuple{\Lambda}}$. In particular, the algebra homomorphism $\sigma : \sqH_\alpha(\Gamma) \to \sqH_{\sigma \cdot \alpha}(\Gamma)$ induces an algebra homomorphism $\sigma^{\tuple{\Lambda}} : \sqH_\alpha^{\tuple{\Lambda}}(\Gamma) \to \sqH_{\sigma \cdot \alpha}^{\sigma \cdot \tuple{\Lambda}}(\Gamma)$.
\end{lemme}

\begin{proof}
We notice that for $\tuple{k} \in K^\alpha$ we have:
\[
\sigma(y_1^{\Lambda_{k_1}} e(\tuple{k})) = y_1^{\Lambda_{k_1}} e(\sigma(\tuple{k})) \in \mathcal{I}_{\sigma \cdot \alpha}^{\sigma \cdot \tuple{\Lambda}},
\]
since $\Lambda_{k_1} = (\sigma \cdot \tuple{\Lambda})_{\sigma(\tuple{k})_1}$.
Hence $\sigma(\mathcal{I}_\alpha^{\tuple{\Lambda}}) \subseteq \mathcal{I}_{\sigma \cdot \alpha}^{\sigma \cdot \tuple{\Lambda}}$ and we have equality by repeating the argument with $\sigma^{-1}$.
\end{proof}

Until the end of this section, we make the following $\sigma$-stability assumption on our weight $\tuple{\Lambda} \in \mathbb{N}^{(K)}$:
\begin{equation}
\label{equation:weight_adapted}
\Lambda_k = \Lambda_{\sigma(k)}, \qquad \text{for all }  k \in K.
\end{equation}
that is, we assume that $\tuple{\Lambda} = \sigma \cdot \tuple{\Lambda}$. Equivalently, the weight $\tuple{\Lambda}$ factors to an element of $\mathbb{N}^{(K / {\sim})}$ (with the notation of \eqref{equation:definition_tilde}). The reader may have noticed the similarity with the equation of Proposition~\ref{proposition:removing_repetitions-_Lambda}. In \textsection\ref{subsection:intertwining} we will explicitly make the link between these two assumptions.

Similarly to \eqref{equation:definition_qH[alpha]}, we define:
\[
\qH_{[\alpha]}^{\tuple{\Lambda}}(\Gamma) \coloneqq \bigoplus_{\beta \in [\alpha]} \qH_\beta^{\tuple{\Lambda}}(\Gamma).
\]
This algebra is the quotient of $\sqH_{[\alpha]}(\Gamma)$ by the two sided ideal
\[
\mathcal{I}_{[\alpha]}^{\tuple{\Lambda}} \coloneqq \bigoplus_{\beta \in [\alpha]} \mathcal{I}_\beta^{\tuple{\Lambda}}
\]
generated by the elements $y_1^{\Lambda_{k_1}} e(\tuple{k})$ for $\tuple{k} \in K^{[\alpha]}$. We deduce from  Lemma~\ref{lemma:sigma_Ilambda} the following statement.
 
\begin{lemme}
\label{lemma:sigma_I[alpha]}
We have $\sigma(\mathcal{I}^{\tuple{\Lambda}}_{[\alpha]}) = \mathcal{I}^{\tuple{\Lambda}}_{[\alpha]}$. Moreover, $\sigma : \qH_{[\alpha]}(\Gamma) \to \qH_{[\alpha]}(\Gamma)$ induces an algebra homomorphism $\sigma^{\tuple{\Lambda}} : \qH_{[\alpha]}^{\tuple{\Lambda}}(\Gamma) \to \qH_{[\alpha]}^{\tuple{\Lambda}}(\Gamma)$.
\end{lemme}

If $\pi_{[\alpha]} : \qH_{[\alpha]}(\Gamma) \twoheadrightarrow \qH_{[\alpha]}^{\tuple{\Lambda}}(\Gamma)$ is the canonical projection, by definition the induced automorphism $\sigma^{\tuple{\Lambda}}$ satisfies:
\begin{equation}
\label{equation:definition_sigma^Lambda}
\sigma^{\tuple{\Lambda}} \circ \pi_{[\alpha]} = \pi_{[\alpha]} \circ \sigma.
\end{equation}
We will often write $\sigma$ as well for the automorphism $\sigma^{\tuple{\Lambda}}$.

\begin{definition}
We define ${\qH_{[\alpha]}^{\tuple{\Lambda}}(\Gamma)}^\sigma$ as the $F$-algebra of the fixed points of $\qH_{[\alpha]}^{\tuple{\Lambda}}(\Gamma)$ under the automorphism $\sigma^{\tuple{\Lambda}}$.
\end{definition}

We recall the notation of \textsection\ref{subsubsection:fixed_affine_subalgebra}.
Since $\tuple{\Lambda}$ satisfies the $\sigma$-stability assumption \eqref{equation:weight_adapted} and considering the canonical map $K^n_\sigma = K^n / {\sim} \to (K/ {\sim})^n$, we may also consider the algebra ${(\qH_{[\alpha]}(\Gamma)^\sigma)}^{\tuple{\Lambda}}$, the quotient of $\qH_{[\alpha]}(\Gamma)^\sigma$ by the two-sided ideal $\mathcal{I}_{[\alpha], \sigma}^{\tuple{\Lambda}}$ generated by the following relations:
\begin{equation}
\label{relation:quiver_cyclo_fixed}
y_1^{\Lambda_{\gamma_1}} e(\gamma) = 0, \qquad \text{for all } \gamma \in K^{[\alpha]}_\sigma.
\end{equation}
In order to give a presentation of $\qH_{[\alpha]}^{\tuple{\Lambda}}(\Gamma)^\sigma$,  we want to prove that this algebra is isomorphic to the following one:
\[
{(\qH_{[\alpha]}(\Gamma)^\sigma)}^{\tuple{\Lambda}} = 
\qH_{[\alpha]}(\Gamma)^\sigma \left/ \mathcal{I}_{[\alpha], \sigma}^{\tuple{\Lambda}}\right.,
\]
for which we know a presentation. Recalling that the characteristic of $F$ does not divide $p$, we can define the following linear map:
\begin{equation}
\label{equation:cyclotomic-case_mu}
\begin{array}{c|ccl}
\mu : & \qH_{[\alpha]}(\Gamma) &\longrightarrow& \qH_{[\alpha]}(\Gamma)
\\
& h & \longmapsto & \frac{1}{p}\sum_{m = 0}^{p-1} \sigma^m(h)
\end{array}.
\end{equation}
We now give a succession of lemmas involving this map $\mu$.

\begin{lemme}
\label{lemma:mu_projection}
The following properties are satisfied by the linear map $\mu$:
\begin{align*}
\mu(\qH_{[\alpha]}(\Gamma)) &= \qH_{[\alpha]}(\Gamma)^\sigma,
\\
\mu(h) &= h, \qquad \text{for all } h \in \qH_{[\alpha]}(\Gamma)^\sigma.
\end{align*}
Moreover, we have:
\[
\mu(\mathcal{I}_{[\alpha]}^{\tuple{\Lambda}}) = \mathcal{I}_{[\alpha]}^{\tuple{\Lambda}} \cap \qH_{[\alpha]}(\Gamma)^\sigma.
\]
\end{lemme}

\begin{proof}
The first two statements follow from $\sigma^p = \mathrm{id}$.
We deduce the last one using Lemma~\ref{lemma:sigma_I[alpha]}.
\end{proof}

\begin{remarque}
The linear map $\mu$ is a linear projection onto the subspace $\qH_{[\alpha]}(\Gamma)^\sigma$.
\end{remarque}

\begin{lemme}
\label{lemma:image_mu_ideal}
We have the following equality:
\[
\mathcal{I}_{[\alpha]}^{\tuple{\Lambda}} \cap \qH_{[\alpha]}(\Gamma)^\sigma = \mathcal{I}_{[\alpha], \sigma}^{\tuple{\Lambda}}.
\]
In particular, $\mu(\mathcal{I}_{[\alpha]}^{\tuple{\Lambda}}) = \mathcal{I}_{[\alpha], \sigma}^{\tuple{\Lambda}}$.
\end{lemme}

\begin{proof}
Since for $\gamma \in K^{[\alpha]}_\sigma$ we have $y_1^{\Lambda_{\gamma_1}} e(\gamma) \in \mathcal{I}_{[\alpha]}^{\tuple{\Lambda}} \cap \qH_{[\alpha]}(\Gamma)^\sigma$, we get $\mathcal{I}_{[\alpha]}^{\tuple{\Lambda}} \cap \qH_{[\alpha]}(\Gamma)^\sigma \supseteq \mathcal{I}_{[\alpha], \sigma}^{\tuple{\Lambda}}$. We now consider an element $h$ of $\mathcal{I}_{[\alpha]}^{\tuple{\Lambda}}$. 
Because of  \eqref{relation:quiver_sum_e(k)}, \eqref{relation:quiver_e(k)e(k')} and Theorem~\ref{theorem:generating_family_quiver_cyclo}, we know that we have:
\[
h = \sum_{w_1, w_2 \in \mathfrak{S}_n} \sum_{\tuple{m}^1, \tuple{m}^2 \in \mathbb{N}^n} \sum_{\tuple{k} \in K^{[\alpha]}} h_{\tuple{k}} \psi_{w_1} y_1^{m^1_1} \cdots y_n^{m^1_n} \left[y_1^{\Lambda_{k_1}} e(\tuple{k})\right] \psi_{w_2} y_1^{m^2_1} \cdots y_n^{m^2_n},
\]
where the $h_{\tuple{k}} \in F$ are some scalars which depend on $\tuple{k}$ and on the other various indices of the sums. For $m \in \{0, \dots, p-1\}$ we have:
\[
\sigma^m(h) = \sum_{w_1, w_2 \in \mathfrak{S}_n}\sum_{\tuple{m}^1, \tuple{m}^2 \in \mathbb{N}^n} \sum_{\tuple{k} \in K^{[\alpha]}} h_{\tuple{k}} \psi_{w_1} y_1^{m^1_1} \cdots y_n^{m^1_n} \left[y_1^{\Lambda_{k_1}} e(\sigma^m(\tuple{k}))\right] \psi_{w_2} y_1^{m^2_1} \cdots y_n^{m^2_n}.
\]
Summing all these equalities from $m = 0$ to $p-1$ and using \eqref{equation:e(gamma)_sum_e(sigma^m(k))}  we get, where $\gamma_{\tuple{k}} \in K^{[\alpha]}_\sigma$ is such that $\tuple{k} \in \gamma_{\tuple{k}}$:
\begin{equation}
\label{equation:cyclotomic_case-to_change_p'}
p \mu(h) =  \sum_{w_1, w_2 \in \mathfrak{S}_n}\sum_{\tuple{m}^1, \tuple{m}^2 \in \mathbb{N}^n} \sum_{\tuple{k} \in K^{[\alpha]}} h_{\tuple{k}} \psi_{w_1} y_1^{m^1_1} \cdots y_n^{m^1_n} \left[y_1^{\Lambda_{k_1}} e(\gamma_{\tuple{k}})\right] \psi_{w_2} y_1^{m^2_1} \cdots y_n^{m^2_n}.
\end{equation}
Since:
\begin{itemize}
\item we have $p \neq 0$;
\item the elements $\psi_{w_1}  y_1^{m^1_1} \cdots y_n^{m^1_n} e(\gamma_{\tuple{k}})$ and $e(\gamma_{\tuple{k}}) \psi_{w_2} y_1^{m^2_1} \cdots y_n^{m^2_n}$ belong to $\qH_{[\alpha]}(\Gamma)^\sigma$;
\item  with $\gamma \coloneqq \gamma_{\tuple{k}}$ we have $\Lambda_{\gamma_1} = \Lambda_{k_1}$ (recall \eqref{equation:weight_adapted});
\end{itemize}
we deduce that $\mu(h) \in \mathcal{I}_{[\alpha], \sigma}^{\tuple{\Lambda}}$. Hence, if in addition $h \in \qH_{[\alpha]}(\Gamma)^\sigma$ then we have $\mu(h) = h$ thus we conclude that $\mathcal{I}_{[\alpha]}^{\tuple{\Lambda}} \cap \qH_{[\alpha]}(\Gamma)^\sigma \subseteq \mathcal{I}_{[\alpha], \sigma}^{\tuple{\Lambda}}$. Finally, we get $\mathcal{I}_{[\alpha]}^{\tuple{\Lambda}} \cap \qH_{[\alpha]}(\Gamma)^\sigma = \mathcal{I}_{[\alpha], \sigma}^{\tuple{\Lambda}}$, and we deduce the last statement from Lemma~\ref{lemma:mu_projection}.
\end{proof}

\begin{lemme}
\label{lemma:good_lifting}
For each $\mathtt{h} \in \qH_{[\alpha]}^{\tuple{\Lambda}}(\Gamma)^\sigma$, there is some $h \in \qH_{[\alpha]}(\Gamma)^\sigma$ such that $\pi_{[\alpha]}(h) = \mathtt{h}$.
\end{lemme}

\begin{proof}
Let $\mathtt{h} \in \qH_{[\alpha]}^{\tuple{\Lambda}}(\Gamma)^\sigma \subseteq  \qH_{[\alpha]}^{\tuple{\Lambda}}(\Gamma)$ and let $h_0 \in \qH_{[\alpha]}(\Gamma)$ be such that $\pi_{[\alpha]}(h_0) = \mathtt{h}$. Since $\mathtt{h}$ is fixed by $\sigma^{\tuple{\Lambda}}$, we have $\sigma(h_0) - h_0 \in \mathcal{I}_{[\alpha]}^{\tuple{\Lambda}}$. Hence, by Lemma~\ref{lemma:sigma_I[alpha]} we obtain:
\[
\sigma^{m'+1}(h_0) - \sigma^{m'}(h_0) \in \mathcal{I}_{[\alpha]}^{\tuple{\Lambda}}, \qquad \text{for all } m' \in \{0, \dots, p-1\},
\]
thus, by summing:
\[
\sigma^m(h_0) - h_0 \in \mathcal{I}_{[\alpha]}^{\tuple{\Lambda}},
\]
for all $m \in \{0, \dots, p-1\}$
(note that this is trivial for $m = 0$). Setting $h \coloneqq \mu(h_0)$ we get $h - h_0 \in \mathcal{I}_{[\alpha]}^{\tuple{\Lambda}}$ thus $\pi_{[\alpha]}(h) = \pi_{[\alpha]}(h_0) = \mathtt{h}$. We conclude since by Lemma~\ref{lemma:mu_projection} we have $h \in \qH_{[\alpha]}(\Gamma)^\sigma$.
\end{proof}

We are now ready to state the main theorem of this section. We recall that the idempotents $\{e(\gamma)\}_\gamma$ of \eqref{equation:generators_quiver_fixed} are indexed by the set $K^{[\alpha]}_\sigma$, which is defined in Definition~\ref{definition:affine_case-Kalphasigma}.

\begin{theoreme}
\label{theorem:quiver_two_subalgebra_same}
The algebras ${\qH_{[\alpha]}^{\tuple{\Lambda}}(\Gamma)}^\sigma$ and ${(\qH_{[\alpha]}(\Gamma)^\sigma)}^{\tuple{\Lambda}}$ are isomorphic. In particular, the generators \eqref{equation:generators_quiver_fixed} together with the relations 
\eqref{relations:affine_quiver_fixed} and \eqref{relation:quiver_cyclo_fixed} give a presentation of ${\qH_{[\alpha]}^{\tuple{\Lambda}}(\Gamma)}^\sigma$.
\end{theoreme}

\begin{proof}
Recalling Corollary~\ref{corollary:presentation_quiver_fixed}, we begin by noticing that the given presentation is a presentation of ${(\qH_{[\alpha]}(\Gamma)^\sigma)}^{\tuple{\Lambda}}$. In particular, we can define a homomorphism of algebras $f : {(\qH_{[\alpha]}(\Gamma)^\sigma)}^{\tuple{\Lambda}} \to {\qH_{[\alpha]}^{\tuple{\Lambda}}(\Gamma)}^\sigma$ by:
\begin{align*}
f(e(\gamma)) &\coloneqq e(\gamma), &&\text{for all } \gamma \in K^{[\alpha]}_\sigma,
\\
f(y_a) &\coloneqq y_a, &&\text{for all } a \in \{1, \dots, n\},
\\
f(\psi_a) &\coloneqq \psi_a, &&\text{for all } a \in \{1, \dots, n-1\}.
\end{align*}
If $\pi_{[\alpha]}^\sigma : \qH_{[\alpha]}(\Gamma)^\sigma \twoheadrightarrow {(\qH_{[\alpha]}(\Gamma)^\sigma)}^{\tuple{\Lambda}}$ is the canonical projection, we have:
\begin{equation}
\label{equation:proof_quiver_sigma_Lambda_commute_easy}
f \circ \pi_{[\alpha]}^\sigma(h) = \pi_{[\alpha]}(h)
\end{equation}
for all $h \in \qH_{[\alpha]}(\Gamma)^\sigma$
(it suffices to check this equality for each generator of $\qH_{[\alpha]}(\Gamma)^\sigma$).
We now want to construct an inverse  ${\qH_{[\alpha]}^{\tuple{\Lambda}}(\Gamma)}^\sigma \to {(\qH_{[\alpha]}(\Gamma)^\sigma)}^{\tuple{\Lambda}}$ for $f$. Note that it is not obvious  since we do not have any presentation of the starting algebra yet. By Lemma~\ref{lemma:mu_projection}, we have a well-defined linear map $\mu_1 :  \qH_{[\alpha]}(\Gamma) \to {(\qH_{[\alpha]}(\Gamma)^\sigma)}^{\tuple{\Lambda}}$ given by $\mu_1 = \pi_{[\alpha]}^\sigma \circ \mu$. By Lemma~\ref{lemma:image_mu_ideal} we have $\ker \mu_1 \supseteq \mathcal{I}_{[\alpha]}^{\tuple{\Lambda}}$, hence we have a well-defined linear map:
\[
\mu_2 : \qH_{[\alpha]}^{\tuple{\Lambda}}(\Gamma) \to  {(\qH_{[\alpha]}(\Gamma)^\sigma)}^{\tuple{\Lambda}},
\]
and by restriction we get a linear map $\overline{\mu} : {\qH_{[\alpha]}^{\tuple{\Lambda}}(\Gamma)}^\sigma \to {(\qH_{[\alpha]}(\Gamma)^\sigma)}^{\tuple{\Lambda}}$. To summarise, we have a commutative diagram \eqref{diagram:overlinemu}.
\begin{equation}
\label{diagram:overlinemu}
\begin{tikzpicture}
[descr/.style={fill=white,inner sep=5pt},
>=angle 90, baseline=(current  bounding  box.center)
]
\matrix (m) [matrix of math nodes, row sep=4em,
column sep=4em]
{\qH_{[\alpha]}(\Gamma)
&
\qH_{[\alpha]}(\Gamma)^\sigma
\\
\qH_{[\alpha]}^{\tuple{\Lambda}}(\Gamma)
&
{(\qH_{[\alpha]}(\Gamma)^\sigma)}^{\tuple{\Lambda}}
\\
\qH_{[\alpha]}^{\tuple{\Lambda}}(\Gamma)^\sigma
&
\\
};

%

\path[->]
(m-1-1)
edge node[auto]
{$\mu$}
(m-1-2);

\path[->>]
(m-1-1)
edge node[auto]
{$\pi_{[\alpha]}$}
(m-2-1);

\path[->>]
(m-1-2)
edge node[auto]
{$\pi_{[\alpha]}^\sigma$}
(m-2-2);

\path[->]
(m-2-1)
edge node[auto]
{$\mu_2$}
(m-2-2);

\path[->]
(m-1-1)
edge node[auto]
{$\mu_1$}
(m-2-2);

\path[right hook->]
(m-3-1)
edge
(m-2-1);

\path[->]
(m-3-1)
edge node[auto]
{$\overline{\mu}$}
(m-2-2);

\end{tikzpicture}
\end{equation}

Let $h \in \qH_{[\alpha]}(\Gamma)^\sigma$. We want to prove the following equality:
\begin{equation}
\label{equation:proof_quiver_sigma_Lambda_commute_hard}
\overline{\mu} \circ \pi_{[\alpha]}(h)
=
\pi_{[\alpha]}^\sigma(h).
\end{equation}
First, by \eqref{equation:definition_sigma^Lambda} we have $\pi_{[\alpha]}(h) \in {\qH_{[\alpha]}^{\tuple{\Lambda}}(\Gamma)}^\sigma$ hence we can evaluate $\overline{\mu}$ at $\pi_{[\alpha]}(h)$. We now use the commutative diagram \eqref{diagram:overlinemu}:
\[
\overline{\mu}\circ\pi_{[\alpha]}(h)
=
\mu_2 \circ \pi_{[\alpha]}(h)
=
\mu_1(h)
=
\pi_{[\alpha]}^\sigma\circ \mu (h),
\]
and we conclude since $\mu(h) = h$ by Lemma~\ref{lemma:mu_projection}.

Finally, let us  prove that $f$ and $\overline{\mu}$ are mutual inverses.
\begin{itemize}
\item Let $\mathtt{h} \in {(\qH_{[\alpha]}(\Gamma)^\sigma)}^{\tuple{\Lambda}}$. If $h \in \qH_{[\alpha]}(\Gamma)^\sigma$ is such that $\pi_{[\alpha]}^\sigma(h) = \mathtt{h}$, we have, using \eqref{equation:proof_quiver_sigma_Lambda_commute_easy} and \eqref{equation:proof_quiver_sigma_Lambda_commute_hard}:
\[
\overline{\mu} \circ f (\mathtt{h}) = \overline{\mu} \circ f \circ \pi_{[\alpha]}^\sigma (h) = \overline{\mu} \circ \pi_{[\alpha]}(h) = \pi_{[\alpha]}^\sigma(h) = \mathtt{h},
\]
hence $\overline{\mu} \circ f$ is the identity of ${(\qH_{[\alpha]}(\Gamma)^\sigma)}^{\tuple{\Lambda}}$.
\item Let $\mathtt{h} \in  {\qH_{[\alpha]}^{\tuple{\Lambda}}(\Gamma)}^\sigma$. By Lemma~\ref{lemma:good_lifting}, we can find $h \in \qH_{[\alpha]}(\Gamma)^\sigma$ such that $\pi_{[\alpha]}(h) = \mathtt{h}$. Using once again   \eqref{equation:proof_quiver_sigma_Lambda_commute_easy} and \eqref{equation:proof_quiver_sigma_Lambda_commute_hard} we get:
\[
f \circ \overline{\mu}(\mathtt{h}) = f \circ \overline{\mu} \circ \pi_{[\alpha]} (h) = f \circ \pi_{[\alpha]}^\sigma(h) = \pi_{[\alpha]}(h) = \mathtt{h},
\]
thus $f \circ \overline{\mu}$ is the identity of ${\qH_{[\alpha]}^{\tuple{\Lambda}}(\Gamma)}^\sigma$.
\end{itemize}

In particular, the algebra homomorphism $f$ is bijective, hence is an algebra isomorphism between ${\qH_{[\alpha]}^{\tuple{\Lambda}}(\Gamma)}^\sigma$ and ${(\qH_{[\alpha]}(\Gamma)^\sigma)}^{\tuple{\Lambda}}$.
\end{proof}

\begin{remarque}
\label{remark:grading_qH_cyclo_fixed}
The grading of Remark~\ref{remark:gradation_qH_fixed} thus gives a grading on $\qH_{[\alpha]}^{\tuple{\Lambda}}(\Gamma)^\sigma$, for which $\sigma^{\tuple{\Lambda}}$ is homogeneous (recall Remark~\ref{remark:sigma_qH_graded}). Moreover, the algebra $\qH_{[\alpha]}^{\tuple{\Lambda}}(\Gamma)^\sigma$ is a graded subalgebra of $\qH_{[\alpha]}^{\tuple{\Lambda}}(\Gamma)$.
\end{remarque}

\begin{remarque}
The proof of Theorem~\ref{theorem:quiver_two_subalgebra_same} uses the assumption that $p$ is invertible in the base field. A characteristic-free (and in a slightly more general setting) version can be found in the author's PhD thesis~\cite[\textsection 1.4]{Ro_phd}.
\end{remarque}

%

\section{The isomorphism of Brundan and Kleshchev}
\label{section:BK}

In this section, we generalise an isomorphism of Brundan and Kleshchev \cite{BrKl} involving $\H_n^{\tuple{\Lambda}}(q, 1)$ to the case of the algebra $\H_n^{\tuple{\Lambda}}(q, \zeta)$.

\subsection{Statement}
\label{subsection:statement}

We consider the quiver $\Gamma_e$ defined as follows:
\begin{itemize}
\item the vertex set is ${\{q^i\}}_{i \in I}$;
\item there is a directed edge from $\vertex{v}$ to $q\vertex{v}$ for each vertex $\vertex{v}$ of $\Gamma_e$.
\end{itemize}
We will often identify the vertex set with $I$ in the canonical way. In particular, if $i$ is a vertex then there is a directed arrow from $i$ to $i + 1$. For $i, i' \in I$, with the notation of Section~\ref{section:CQHA} we thus have:
\begin{align*}
i \to i' &\iff [i' = i + 1 \text{ and } i \neq i' + 1],
\\
i \leftarrow i' &\iff [i = i' + 1 \text{ and } i' \neq i + 1],
\\
i \leftrightarrows i' &\iff [i = i' + 1 \text{ and } i' = i + 1],
\\
i \nrelbar i' &\iff i \neq i', i'\pm 1.
\end{align*}
The quiver $\Gamma_e$ is the cyclic quiver with $e$ vertices if $e < \infty$, and a two-sided infinite line if $e = \infty$: we give some examples in Figure~\ref{figure:gamma_e}, where we used the identification between  the vertex set of $\Gamma_e$ and $I$.
\begin{figure}[h]
\centering
\begin{tabular}{lM{10cm}}
Quiver $\Gamma_2$
&
$0 \leftrightarrows 1$
\\
\\
Quiver $\Gamma_4$
&
\begin{tikzpicture}[>=angle 90]
\node (0) at (0, 1) {$0$};
\node (1) at (1, 1) {$1$};
\node (2) at (1, 0) {$2$};
\node (3) at (0, 0) {$3$};

\draw[->] (0) -- (1);
\draw[->] (1) -- (2);
\draw[->] (2) -- (3);
\draw[->] (3) -- (0);
\end{tikzpicture}
\\
\\
Quiver $\Gamma_{\infty}$
&
\begin{tikzpicture}[>=angle 90]
\node (-3) at (-4, 0) {$\cdots$};
\node (-2) at (-2.6, 0) {$-2$};
\node (-1) at (-1.2, 0) {$-1$};
\foreach \i in {0,1,2}
	\node (\i) at (\i, 0) {$\i$};
\node (3) at (3.2, 0) {$\cdots$};
\foreach \i [count=\j from -2] in {-3,...,2}
	\draw[->] (\i) -- (\j);
\end{tikzpicture}
\end{tabular}
\caption{Three examples of quivers $\Gamma_e$}
\label{figure:gamma_e}
\end{figure}
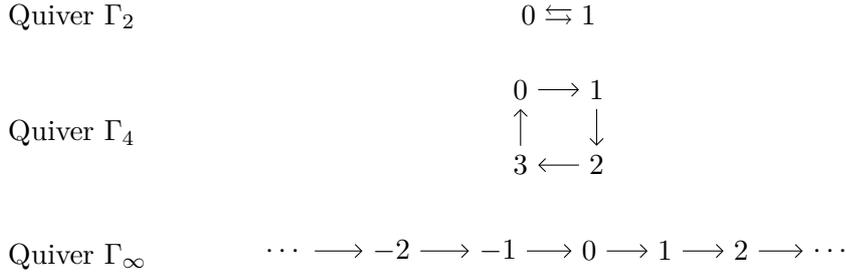

We recall the notation $p'$ and $J'$ introduced at \textsection\ref{subsection:removing_repetitions}. We set $K \coloneqq I \times J'$.
Let us consider $p'$ non-zero elements $v_1, \dots, v_{p'}$ of $F$ which lie in distinct orbits under the action of $\langle q \rangle$ on $F^\times$, that is, for any $k \neq l$ we have:
\begin{equation}
\label{equation:isom_BK-condition_v}
\frac{v_k}{v_l} \notin \langle q \rangle.
\end{equation}
We then consider the quiver $\Gamma$ defined as follows:
\begin{itemize}
\item the vertex set is $\V \coloneqq \{v_j q^i\}_{i \in I, j \in J'}$;
\item there is a directed edge from $\vertex{v}$ to $q\vertex{v}$ for each vertex $\vertex{v}$ of $\Gamma$.
\end{itemize}

Since the $v_k$ lie in different $q$-orbits, the vertex set $\V$ of $\Gamma$ can be identified with $K = I \times J'$.  More precisely, we have the following decomposition:
\begin{equation}
\label{equation:isom_BK-identifying_Gamma}
\V = \bigsqcup_{j \in J'} {\{v_j q^i\}}_{i \in I}.
\end{equation}
Since:
\begin{itemize}
\item the subquiver of $\Gamma$ with vertex set ${\{v_j q^i\}}_{i \in I}$ is a copy of $\Gamma_e$;
\item for $j \neq j' \in J'$, there is no arrow between any element of ${\{v_j q^i\}}_{i \in I}$ and  ${\{v_{j'} q^i\}}_{i \in I}$;
\item the set $J'$ has cardinality $p'$;
\end{itemize}
we conclude from \eqref{equation:isom_BK-identifying_Gamma} that $\Gamma$ is exactly $p'$ disjoint copies of $\Gamma_e$. In particular, the quiver $\Gamma$ is loop-free and has no multiple edges.

As a consequence, we will often write $(i, j) \in I \times J'$ for the vertex $v_j q^i  \in V$ of $\Gamma$. For any $i, i' \in I$ and $j, j' \in J'$, what precedes ensures that the vertices $(i, j)$ and $(i', j')$ are in a same copy of $\Gamma_e$ if and only if $j = j'$. Further, there is a directed edge from $(i, j)$ to $(i', j')$ if and only if $j = j'$ and there is a directed edge in $\Gamma_e$ from $i$ to $i'$.
We give some examples of quivers $\Gamma$ in Figure~\ref{figure:Gamma}, where, for aesthetic reasons, we write $i_j$ instead of $v_j q^i$. We also recall from Lemma~\ref{lemma:introduction-p'_gcd} that $p' = \frac{p}{\mathrm{gcd}(p, e)}$ if $e < \infty$ and $p' = p$ if $e = \infty$.
\begin{figure}[h]
\centering
\begin{tabular}{lM{9.5cm}}
Case $(e, p) = (2, 3)$
&
$0_1 \leftrightarrows 1_1 \qquad 0_2 \leftrightarrows 1_2 \qquad 0_3 \leftrightarrows 1_3$
\\
\\
Case $(e, p) = (2, 6)$
&
$0_1 \leftrightarrows 1_1 \qquad 0_2 \leftrightarrows 1_2 \qquad 0_3 \leftrightarrows 1_3$
\\
\\
Case $(e, p) = (\infty, 2)$
&
\begin{tikzpicture}[>=angle 90]
\node (-3) at (-3.2, 1) {$\cdots$};
\node (-2) at (-1.7, 1) {$-2_1$};
\node (-1) at (-.2, 1) {$-1_1$};
\node (0) at (1.2, 1) {$0_1$};
\node (1) at (2.4, 1) {$1_1$};
\node (2) at (3.6, 1) {$2_1$};
\node (3) at (4.9, 1) {$\cdots$};
\foreach \i [count=\j from -2] in
 {-3,...,2}
	\draw[->] (\i) -- (\j);
\node (4) at (-3.2, 0) {$\cdots$};
\node (5) at (-1.7, 0) {$-2_2$};
\node (6) at (-.2, 0) {$-1_2$};
\node (7) at (1.2, 0) {$0_2$};
\node (8) at (2.4, 0) {$1_2$};
\node (9) at (3.6, 0) {$2_2$};
\node (10) at (4.9, 0) {$\cdots$};
\foreach \i [count=\j from 5] in
 {4,...,9}
	\draw[->] (\i) -- (\j);
\end{tikzpicture}
\end{tabular}
\caption{Three examples of quivers $\Gamma$}
\label{figure:Gamma}
\end{figure} 

Now let $\tuple{\Lambda} = (\Lambda_k)_{k \in K} \in \mathbb{N}^{(K)}$ be a weight of level $r$. Mimicking the definition of $\sH_n^{\tuple{\Lambda}}(q, \zeta)$,
 let us choose a tuple $\tuple{u} \in {(F^\times)}^r$ which is given by exactly $\Lambda_{i, j}$ copies of $v_j q^i$ for each $(i, j) \in I \times J'$ and set $\sH_n^{\tuple{\Lambda}}(q, \tuple{v}) \coloneqq \sH_n(q, \tuple{u})$. As a result, the relation \eqref{relation:hecke_cyclo_S_general} in $\sH_n(q, \tuple{u})$ is:
\begin{equation}
\label{equation:isom_BK-cyclotomic_relation}
\prod_{i \in I} \prod_{j \in J'} {(S - v_j q^i)}^{\Lambda_{i, j}} = 0.
\end{equation}

The remaining part of this section is devoted to the proof of the following theorem.

\begin{theoreme}
\label{theorem:BK_generalised}
There is an explicit $F$-algebra isomorphism:
\[
\sH_n^{\tuple{\Lambda}}(q, \tuple{v}) \simeq \sqH_n^{\tuple{\Lambda}}(\Gamma).
\]
\end{theoreme}
Brundan and Kleshchev \cite{BrKl} proved Theorem~\ref{theorem:BK_generalised} for $p = 1$: in that case, we have $p' = 1$, the tuple $\tuple{v}$ has only one component (that can be taken equal to 1) and $\Gamma = \Gamma_e$. We will see that the same argument proves the general case. Such an isomorphism, for $e < \infty$, was already obtained by Rouquier \cite[Corollary 3.20]{Rou}.

\subsection{Candidate homomorphisms}
\label{subsection:candidate_homomorphisms}
We recall that $V \simeq K = I \times J'$, together with the definitions $X_1 \coloneqq S$ and $q X_{a+1} \coloneqq T_a X_a T_a$ from \eqref{equation:definition_Xa+1}. To prove Theorem~\ref{theorem:BK_generalised}, it suffices to give an isomorphism between $\sH_\alpha^{\tuple{\Lambda}}(q, \tuple{v})$ (see \eqref{equation:candididate_homomorphism-H_alpha}) and $\sqH_\alpha^{\tuple{\Lambda}}(\Gamma)$   for any $\alpha \comp_{K} n$.
Let $M$ be a finite-dimensional $\sH_n^{\tuple{\Lambda}}(q, \tuple{v})$-module.

\begin{lemme}
\label{lemma:eigenvalues_Xa}
For any $a \in \{1, \dots, n\}$, the eigenvalues of $X_a$ on $M$ are of the form $v_j q^i$ for $i \in I$ and $j \in J'$.
\end{lemme}

\begin{proof}
The statement is of course true for $a = 1$ by \eqref{relation:hecke_cyclo_S}. By induction, using \cite{ArKo} or \cite[Lemma 4.7]{Gr} we know that any eigenvalue of $X_{a+1}$ differs from an eigenvalues of $X_a$ by a power of $q$.
\end{proof}

Hence, as the elements $X_1, \dots, X_n$ pairwise commute, we can write $M$ as a direct sum of generalised simultaneous eigenspaces:
\[
M = \bigoplus_{\tuple{k} \in K^n} M(\tuple{k}),
\]
where $M(\tuple{k}) = M(\tuple{i}, \tuple{j})$ is defined by, for $\tuple{k} = (\tuple{i}, \tuple{j}) \in K^n \simeq I^n \times {J'}^n$:
\[
M(\tuple{i}, \tuple{j}) \coloneqq \left\lbrace m \in M : {(X_a - v_{j_a} q^{i_a})}^N m = 0 \text{ for all } 1 \leq a \leq n\right\},
\]
where $N \gg 0$. Note  that all but finitely many $M(\tuple{k})$ are reduced to $\{0\}$. We now consider the family $\{e(\tuple{k})\}_{\tuple{k} \in K^n}$ of projections associated with the  decomposition above. In particular:
\begin{itemize}
\item we have $e(\tuple{k}) e(\tuple{k}')= \delta_{\tuple{k}, \tuple{k}'} e(\tuple{k})$;
\item we have $\sum_{\tuple{k} \in K^n} e(\tuple{k}) = \mathrm{id}$ (this is a finite sum since all but finitely many $e(\tuple{k})$ are zero);
\item we have $e(\tuple{k})M = M(\tuple{k})$.
\end{itemize}

\begin{remarque}
We already used the notation $e(\tuple{k})$ for some generators of $\qH_n^{\tuple{\Lambda}}(\Gamma)$. This abuse of notation will be justified by the proof of Theorem~\ref{theorem:BK_generalised}, where we prove that these elements can be identified.
\end{remarque}

Since $e(\tuple{k})$ is a polynomial in $X_1, \dots, X_n$ we have  $e(\tuple{k}) \in \sH_n^{\tuple{\Lambda}}(q, \tuple{v})$. If now $\alpha \comp_{K} n$ is a $K$-composition of $n$, the following element:
\[
e(\alpha) \coloneqq \sum_{\tuple{k} \in K^\alpha} e(\tuple{k}) \in \sH_n^{\tuple{\Lambda}}(q, \tuple{v}),
\]
is a central idempotent (the reader should compare this definition to \eqref{equation:definition_e(alpha)_QH}). We thus get a subalgebra:
\begin{equation}
\label{equation:candididate_homomorphism-H_alpha}
\sH_\alpha^{\tuple{\Lambda}}(q, \tuple{v}) \coloneqq e(\alpha) \sH_n^{\tuple{\Lambda}}(q, \tuple{v}).
\end{equation}

\begin{remarque}
The subalgebra $\sH_\alpha^{\tuple{\Lambda}}(q, \tuple{v})$ is either $\{0\}$ or a \emph{block} of $\sH_n^{\tuple{\Lambda}}(q, \tuple{v})$ (see \cite{LyMa}). This block has unit $e(\alpha)$.
\end{remarque}

Recall that the elements $y_a \in \sqH_\alpha^{\tuple{\Lambda}}(\Gamma)$ for $1 \leq a \leq n$ are nilpotent (Lemma~\ref{lemma:ya_nilpotent}). Hence, each power series $f(y_1, \dots, y_n) \in F[[y_1, \dots, y_n]]$ in these elements is a well-defined element of $\sqH_\alpha^{\tuple{\Lambda}}(\Gamma)$. In particular, for $a \in \{1, \dots, n-1\}$ and $\tuple{k} \in K^\alpha$ the following power series is well-defined in $\sqH_\alpha^{\tuple{\Lambda}}(\Gamma)$:
\begin{equation}
\label{equation:candidate_homomorphisms-P_a}
P_a(\tuple{k}) \coloneqq \begin{cases}
1 & \text{if } k_a = k_{a+1},
\\
(1 - q){(1 - y_a(\tuple{k})y_{a+1}(\tuple{k})^{-1})}^{-1} & \text{if } k_a \neq k_{a+1},
\end{cases} \in F[[y_a, y_{a+1}]],
\end{equation}
where:
\begin{equation}
\label{equation:definition_ya(k)}
y_a(\tuple{k}) \coloneqq v_{j_a} q^{i_a}(1 - y_a) \text{ if }\tuple{k} = (\tuple{i}, \tuple{j}).
\end{equation}
For $k = (i, j) \in K = I \times J'$, we define:
\begin{equation}
\label{equation:definition_zetak_qk}
\begin{gathered}
q^k \coloneqq q^i, \qquad q^{-k} \coloneqq q^{-i},
\\
v_k \coloneqq v_j, \qquad v_{-k} \coloneqq v_j^{-1},
\end{gathered}
\end{equation}
in particular we obtain $y_a(\tuple{k}) = v_{k_a} q^{k_a} (1 - y_a)$ for any $\tuple{k} \in K^n$.
Note that:
\begin{align*}
1 - y_a(\tuple{k})y_{a+1}(\tuple{k})^{-1}
&=
\frac{y_{a+1}(\tuple{k}) - y_a(\tuple{k})}{y_{a+1}(\tuple{k})}
\\
&= \frac{(v_{k_{a+1}} q^{k_{a+1}} - v_{k_a} q^{k_a}) + v_{k_a} q^{k_a} y_a - v_{k_{a+1}}q^{k_{a+1}}y_{a+1}}{y_{a+1}(\tuple{k})};
\end{align*}
thus, by \eqref{equation:isom_BK-condition_v} we know that this expression is indeed invertible when $k_a \neq k_{a+1}$.

For $(w, f) \in \mathfrak{S}_n \times F[[y_1, \dots, y_n]]$, we denote by $f^w \in F[[y_1, \dots, y_n]]$ the usual right action of $w$ on $f$. For instance, if $w = \tau$ is a transposition then $f^\tau(y_1, \dots, y_n) = f(y_{\tau(1)}, \dots, y_{\tau(n)})$.
Let us give a lemma involving this action (see, for instance, \cite[(2.6)]{BrKl}).

\begin{lemme}
\label{lemma:f_psia}
For any $f \in F[[y_1, \dots, y_n]], a \in \{1, \dots, n-1\}$ and $\tuple{k} \in K^\alpha$ we have:
\[
f \psi_a e(\tuple{k}) = \begin{cases}
\psi_a f^{s_a} e(\tuple{k})  + \partial_a(f) e(\tuple{k}) & \text{if } k_a = k_{a+1},
\\
\psi_a f^{s_a} e(\tuple{k}) & \text{if } k_a \neq k_{a+1},
\end{cases}
\]
where $\partial_a(f) \coloneqq \frac{f^{s_a} - f}{y_a - y_{a+1}} \in F[[y_1, \dots, y_n]]$.
\end{lemme}

\begin{proof}
This is a consequence of \eqref{relation:quiver_psia_yb}, \eqref{relation:quiver_psia_ya+1} and \eqref{relation:quiver_ya+1_psia}.
\end{proof}

We say that a family $\{Q_a(\tuple{k})\}_{a \in \{1, \dots, n-1\}, \tuple{k} \in K^\alpha}$ of elements of $F[[y_1,\dots,y_n]]$ satisfies the property $(\mathrm{BK})$ if:
\begin{gather}
\label{equation:Qa_invertible}
Q_a(\tuple{k}) \text{ is an invertible element of } F[[y_a, y_{a+1}]],
\\
\label{equation:Qa_=}
Q_a(\tuple{k}) = 1 - q + q y_{a+1} - y_a \qquad \text{if } k_a = k_{a+1},
\end{gather}
\begin{subnumcases}{\label{equation:QaQasa} Q_a(\tuple{k}) Q_a(s_a \cdot \tuple{k})^{s_a} =}
\label{equation:Qa_nrelbar}
(1 - P_a(\tuple{k}))(q + P_a(\tuple{k})) & if $k_a \nrelbar k_{a+1}$,
\\
\label{equation:Qa_->}
\frac{(1 - P_a(\tuple{k}))(q + P_a(\tuple{k}))}{y_{a+1} - y_a} & if $k_a \to k_{a+1}$,
\\
\label{equation:Qa_<-}
\frac{(1 - P_a(\tuple{k}))(q + P_a(\tuple{k}))}{y_a - y_{a+1}} & if $k_a \leftarrow k_{a+1}$,
\\
\label{equation:Qa_<->}
\frac{(1 - P_a(\tuple{k}))(q + P_a(\tuple{k}))}{(y_{a+1} - y_a)(y_a - y_{a+1})} & if $k_a \leftrightarrows k_{a+1}$,
\end{subnumcases}
\begin{equation}
\label{equation:Qa_odd}
Q_{a+1}(s_{a+1}s_a \cdot \tuple{k})^{s_a} = Q_a(s_a s_{a+1} \cdot \tuple{k})^{s_{a+1}}.
\end{equation}

We can now give the key of Theorem~\ref{theorem:BK_generalised}.

\begin{theoreme}
\label{theorem:BK_generalised_block}
Let $\{Q_a(\tuple{k})\}_{a \in \{1, \dots, n-1\}, \tuple{k} \in K^\alpha}$ be a family of elements of $F[[y_1, \dots, y_n]]$ which satisfies $(\mathrm{BK})$. There exist unique $F$-algebra homomorphisms $f : \sH^{\tuple{\Lambda}}_\alpha(q, \tuple{v}) \to \sqH_\alpha^{\tuple{\Lambda}}(\Gamma)$
and $g : \sqH_\alpha^{\tuple{\Lambda}}(\Gamma) \to \sH^{\tuple{\Lambda}}_\alpha(q, \tuple{v})$
such that:
\begin{gather*}
f(X_a) \coloneqq \sum_{\tuple{k} \in K^\alpha} y_a(\tuple{k}) e(\tuple{k}),
\\
f(T_a) \coloneqq \sum_{\tuple{k} \in K^\alpha} (\psi_a Q_a(\tuple{k}) - P_a(\tuple{k})) e(\tuple{k}),
\end{gather*}
and, recalling \eqref{equation:definition_zetak_qk}:
\begin{gather*}
g(e(\tuple{k})) \coloneqq e(\tuple{k}),
\\
g(y_a) \coloneqq \sum_{\tuple{k} \in K^\alpha} (1 - v_{-k_a} q^{-k_a} X_a) e(\tuple{k}),
\\
g(\psi_a) \coloneqq \sum_{\tuple{k} \in K^\alpha} (T_a + P_a(\tuple{k})) Q_a(\tuple{k})^{-1} e(\tuple{k}).
\end{gather*}
Moreover, these homomorphisms are inverse to each other, hence $\H_\alpha^{\tuple{\Lambda}}(q, \tuple{v}) \simeq \qH_\alpha^{\tuple{\Lambda}}(\Gamma)$.
\end{theoreme}
We will explain at the beginning of \textsection\ref{subsubsection:beta_homomorphism} how the elements $P_a(\tuple{k})$ and $Q_a(\tuple{k})$ are considered as elements of $\sH_\alpha^{\tuple{\Lambda}}(q, \tuple{v})$.
We  note that there exist such families $\{Q_a(\tuple{k})\}_{a, \tuple{k}}$, see \textsection\ref{subsection:nice_family} for further details.  


\subsection{Proof of Theorems \ref{theorem:BK_generalised} and \ref{theorem:BK_generalised_block}}

In this subsection, we first check that the maps of Theorem~\ref{theorem:BK_generalised_block} indeed define algebras homomorphisms: we check that the different defining relations \eqref{relation:hecke_ordre}--\eqref{relation:hecke_braid3}, \eqref{relation:hecke_cyclo_S} (for $f$) and
\eqref{relations:affine_quiver}, \eqref{relation:quiver_cyclo} (for $g$) are satisfied. The proof is exactly as in \cite[Section 4]{BrKl}: we will only give some details when some $v_j$ are involved. The remaining parts of the argument require only notational changes from~\cite{BrKl}.

\subsubsection{The map \texorpdfstring{$f$}{f} is a homomorphism}
\label{subsubsection:alpha_homomorphism}

We prove that the images of the generators of $\sH_n^{\tuple{\Lambda}}(q, \tuple{v})$ by  $f$ satisfy the defining relations.

The proof of the quadratic relation \eqref{relation:hecke_ordre} is exactly the same as the one for \cite[Theorem 4.3]{BrKl}. Namely, it suffices to check that for any $\tuple{k} \in K^\alpha$ we have $f(T_a)^2 e(\tuple{k}) = (q-1)f(T_a) e(\tuple{k}) + q e(\tuple{k})$, and the result follows since $\sum_{\tuple{k} \in K^\alpha} e(\tuple{k}) = 1$ in $\sqH_\alpha^{\tuple{\Lambda}}(\Gamma)$. 

The equality $f(X_1)f(X_2) = f(X_2) f(X_1)$ is clear. Hence, to check the length $4$-braid relation \eqref{relation:hecke_ST1ST1} it suffices to prove that $qf(X_2) = f(T_1) f(X_1) f(T_1)$. We will in fact prove that for any $a \in \{1, \dots, n-1\}$:
\begin{equation}
\label{equation:alpha(Xa+1)}
qf(X_{a+1}) = f(T_a) f(X_a) f(T_a).
\end{equation}
Since we have just checked the relation \eqref{relation:hecke_ordre} for $f(T_a)$, it suffices to prove that for any $\tuple{k} \in K^\alpha$:
\[
f(X_a) f(T_a) e(\tuple{k}) = (f(T_a) + 1 - q)f(X_{a+1}) e(\tuple{k}).
\]
Once again, the rest of the proof is exactly the same as in the corresponding part of the proof of \cite[Theorem 4.3]{BrKl}. We write down here some of the details since we have to add some $v_j$ in the calculations. We have:
\begin{align*}
X_a T_a e(\tuple{k}) &= \left(y_a(s_a \cdot \tuple{k}) \psi_a Q_a(\tuple{k}) - y_a (\tuple{k}) P_a(\tuple{k})\right) e(\tuple{k})
\\
&=
\left(\psi_a y_{a+1}(\tuple{k}) Q_a(\tuple{k}) + \delta_{k_a, k_{a+1}} v_{k_a} q^{k_a} Q_a(\tuple{k}) - y_a(\tuple{k}) P_a(\tuple{k})\right) e(\tuple{k}),
\end{align*}
and:
\[
(T_a + 1 - q)X_{a+1} e(\tuple{k}) = \left(\psi_a Q_a(\tuple{k}) - P_a(\tuple{k}) + 1 - q\right) y_{a+1}(\tuple{k}) e(\tuple{k}).
\]
Considering the two cases $k_a \neq k_{a+1}$ and $k_a = k_{a+1}$ separately and using \eqref{equation:candidate_homomorphisms-P_a}, \eqref{equation:definition_ya(k)} and \eqref{equation:Qa_=}, we can easily prove that the two above quantities are equal.

The commutation relations \eqref{relation:hecke_STa} and \eqref{relation:hecke_braid2} are straightforward from the defining relations in $\sqH_\alpha^{\tuple{\Lambda}}(\Gamma)$, and for \eqref{relation:hecke_braid3} we can reproduce the corresponding part of the proof of \cite[Theorem 4.3]{BrKl}

Finally, let us prove that the cyclotomic relation \eqref{relation:hecke_cyclo_S} is satisfied, that is:
\[
\prod_{k \in K} {\left(f(X_1) - v_k q^k\right)}^{\Lambda_k} = 0.
\]
We have, using \eqref{relation:quiver_sum_e(k)} and  \eqref{relation:quiver_e(k)e(k')}:
\begin{align*}
\prod_{k \in K} {\left(f(X_1) - v_k q^k\right)}^{\Lambda_k}
&=
\prod_{k \in K} \left[\sum_{\tuple{k} \in K^\alpha} \left(v_{k_1} q^{k_1}(1 - y_1) - v_k q^k\right)e(\tuple{k})\right]^{\Lambda_k}
\\
&=
\prod_{k \in K} \left[\sum_{\tuple{k} \in K^\alpha} \left(v_{k_1} q^{k_1}(1 - y_1) - v_k q^k\right)^{\Lambda_k} e(\tuple{k})\right]
\\
&=
\sum_{\tuple{k} \in K^\alpha} \prod_{k \in K} \left[\left(v_{k_1} q^{k_1} (1 - y_1) - v_k q^k\right)^{\Lambda_k} e(\tuple{k})\right].
\end{align*}
By \eqref{relation:quiver_cyclo}, for $\tuple{k} \in K^\alpha$ the term for $k = k_1$ vanishes, hence we get the result.

\bigskip
To conclude, the map $f : \sH_n^{\tuple{\Lambda}}(q, \tuple{v}) \to \sqH_\alpha^{\tuple{\Lambda}}(\Gamma)$ defined on the generators $X_1, T_1, \dots, T_{n-1}$ yields a homomorphism of algebra. By restriction, we get an algebra homomorphism $f : \sH_\alpha^{\tuple{\Lambda}}(q, \tuple{v}) \to \sqH_\alpha^{\tuple{\Lambda}}(\Gamma)$. In particular, the image of $X_a$ for $a > 1$ is the one given in Theorem~\ref{theorem:BK_generalised_block}, thanks to \eqref{equation:definition_Xa+1} and \eqref{equation:alpha(Xa+1)}.

\subsubsection{The map \texorpdfstring{$g$}{g} is a homomorphism}
\label{subsubsection:beta_homomorphism}

In this paragraph, for any $m \in \sqH_\alpha^{\tuple{\Lambda}}(\Gamma)$ we also write $m \coloneqq g(m) \in \sH_\alpha^{\tuple{\Lambda}}(q, \tuple{v})$. In particular, we have:
\[
y_a = \sum_{\tuple{k} \in K^\alpha} (1 - v_{-k_a} q^{-k_a} X_a) e(\tuple{k}) \in \sH_\alpha^{\tuple{\Lambda}}(q, \tuple{v}),
\]
thus we can consider the power series $P_a(\tuple{k})$ and $Q_a(\tuple{k})$ as elements of $\sH_\alpha^{\tuple{\Lambda}}(q, \tuple{v})$, namely:
\[
\psi_a = \sum_{\tuple{k} \in K^\alpha} (T_a + P_a(\tuple{k}))Q_a(\tuple{k})^{-1} e(\tuple{k}) \in \sH_\alpha^{\tuple{\Lambda}}(q, \tuple{v}).
\]

Following Lusztig, define the following ``intertwining element'' in $\sH_\alpha^{\tuple{\Lambda}}(q, \tuple{v})$ for $a \in \{1, \dots, n-1\}$ by:
\[
\Phi_a \coloneqq T_a + (1-q) \sum_{\substack{\tuple{k} \in K^\alpha \\ k_a \neq k_{a+1}}} {(1 - X_a X_{a+1}^{-1})}^{-1} e(\tuple{k}) + \sum_{\substack{\tuple{k} \in K^\alpha \\ k_a = k_{a+1}}} e(\tuple{k}),
\]
where ${(1 - X_a X_{a+1}^{-1})}^{-1} e(\tuple{k})$ denotes the inverse of $(1 - X_a X_{a+1}^{-1})e(\tuple{k})$ in $e(\tuple{k}) \sH_\alpha^{\tuple{\Lambda}}(q, \tuple{v})e(\tuple{k})$. Noticing that $y_a(\tuple{k})e(\tuple{k}) = X_a e(\tuple{k})$, we can check the following equality:
\[
\Phi_a = \sum_{\tuple{k} \in K^\alpha} (T_a + P_a(\tuple{k}))e(\tuple{k}).
\]
We can give an analogue of \cite[Lemma 4.1]{BrKl}. Once again, we just have to write $a$ (respectively $\tuple{k}, k$) instead of their $r$ (resp. $\tuple{i}, i$), both in the statements and the proofs. Among all the relations in the lemma, we will make here an explicit use of the following one:
\begin{equation}
\label{equation:lemme_Phia}
X_{a+1} \Phi_a e(\tuple{k}) = \begin{cases}
\Phi_a X_a e(\tuple{k}) & \text{if } k_a \neq k_{a+1},
\\
\Phi_a X_a e(\tuple{k}) + (qX_{a+1} - X_a)e(\tuple{k}) & \text{if } k_a = k_{a+1}.
\end{cases}
\end{equation}

We now check the different relations of $\sqH_\alpha^{\tuple{\Lambda}}(\Gamma)$. Relations \eqref{relation:quiver_sum_e(k)}--\eqref{relation:quiver_psia_psib}, \eqref{relation:quiver_psia^2}--\eqref{relation:quiver_tresse} and \eqref{relation:quiver_cyclo} follow as in the corresponding part of the proof of \cite[Theorem 4.2]{BrKl}.
To check \eqref{relation:quiver_ya+1_psia}, again we just follow the corresponding part of the proof of \cite[Theorem 4.2]{BrKl}, but we need to add some $v_j$'s. We have:
\[
y_{a+1}\psi_a e(\tuple{k}) = (1 - v_{-k_a} q^{-k_a} X_{a+1})\Phi_a Q_a(\tuple{k})^{-1} e(\tuple{k}).
\]
If $k_a \neq k_{a+1}$, using \eqref{equation:lemme_Phia} we get:
\[
y_{a+1} \psi_a e(\tuple{k})
=
\Phi_a Q_a(\tuple{k})^{-1} (1 - v_{-k_a} q^{-k_a} X_a)e(\tuple{k})
=
\psi_a y_a e(\tuple{k}),
\]
whereas if $k_a= k_{a+1}$ we obtain:
\begin{align*}
y_{a+1} \psi_a e(\tuple{k})
&=
(1 - v_{-k_a}q^{-k_a}X_{a+1})(T_a + 1) Q_a(\tuple{k})^{-1} e(\tuple{k})
\\
&=
\left((T_a + 1) (1 - v_{-k_a} q^{-k_a}X_a) + v_{-k_a} q^{-k_a} X_a - v_{-k_a} q^{1 - k_a} X_{a+1}\right) Q_a(\tuple{k})^{-1} e(\tuple{k})
\\
&=
(\psi_a y_a + 1)e(\tuple{k}),
\end{align*}
since $(v_{-k_a} q^{-k_a} X_a - v_{-k_a}q^{1-k_a}X_{a+1})e(\tuple{k}) = Q_a(\tuple{k})e(\tuple{k})$. The proof of 
\eqref{relation:quiver_psia_ya+1} is similar.

\subsubsection{Conclusion}
\label{subsubsection:conclusion}

As in \cite[Lemma 3.4]{BrKl}, we have:
\[
f(e(\tuple{k})) = e(\tuple{k}) \in \sqH_\alpha^{\tuple{\Lambda}}(\Gamma),
\]
for all $\tuple{k} \in K^\alpha$.
It is now an easy exercise to show that $f \circ g$ is the identity of $\sqH_\alpha^{\tuple{\Lambda}}(\Gamma)$, and then that $g \circ f$ is the identity of $\sH_\alpha^{\tuple{\Lambda}}(q, \tuple{v})$. Hence, the homomorphisms $f$ and $g$ are inverse isomorphisms and Theorem~\ref{theorem:BK_generalised_block} is proved.
Summing the $F$-isomorphism $\sH_\alpha^{\tuple{\Lambda}}(q, \tuple{v}) \simeq \sqH_\alpha^{\tuple{\Lambda}}(\Gamma)$ over all $\alpha \comp_{K} n$, we thus get the statement of Theorem~\ref{theorem:BK_generalised}. Note that since $\sH_\alpha^{\tuple{\Lambda}}(q, \tuple{v})$ is zero for all but finitely many $\alpha$, the same thing happens for $\sqH_\alpha^{\tuple{\Lambda}}(\Gamma)$. In particular, the direct sum:
\[
\sqH_n^{\tuple{\Lambda}}(\Gamma) = \bigoplus_{\alpha \comp_{K} n} \sqH_\alpha^{\tuple{\Lambda}}(\Gamma),
\]
has a finite number of non-vanishing terms.

\subsection{An unexpected corollary}
\label{subsection:unexpected_corollary}
For $j \in J'$, let us write $\tuple{\Lambda}^j$ for the restriction of $\tuple{\Lambda}$ to $I \times \{j\} \simeq I$.
Since $\Gamma$ is given by $p'$ disjoint copies of the quiver $\Gamma_e$, we know from \cite[Theorem 6.30]{Ro} that there is an algebra isomorphism:
\[
\sqH_n^{\tuple{\Lambda}}(\Gamma) \simeq \bigoplus_{\lambda \comp_{J'} n} \mathrm{Mat}_{m_\lambda} \left( \sqH_{\lambda_1}^{\tuple{\Lambda}^1}(\Gamma_e) \otimes \dots \otimes \sqH_{\lambda_{p'}}^{\tuple{\Lambda}^{p'}}(\Gamma_e)\right),
\]
where $m_\lambda \coloneqq \frac{n!}{\lambda_1 ! \cdots \lambda_{p'} !}$. For any $j \in J'$, we set:
\[
\sH_{\lambda_j}^{\tuple{\Lambda}^j}(q) \coloneqq \sH_{\lambda_j}^{\tuple{\Lambda}^j}(q, \tuple{v}_{\text{triv}}),
\]
where $\tuple{v}_{\text{triv}}$ has only one coordinate, equal to $1$.
In particular,  we saw from Theorem~\ref{theorem:BK_generalised} or \cite{BrKl} that we have the $F$-isomorphism $\sH_{\lambda_j}^{\tuple{\Lambda}^j}(q) \simeq \sqH_{\lambda_j}^{\tuple{\Lambda}^j}(\Gamma_e)$. We deduce the following result.

\begin{theoreme}
\label{theorem:unexpected_corollary-oplus}
Let $\tuple{v} \in (F^\times)^{p'}$ as in \textsection\ref{subsection:statement}.
We have an (explicit) $F$-algebra isomorphism:
\[
\sH_n^{\tuple{\Lambda}}(q, \tuple{v}) \simeq \bigoplus_{\lambda \comp_{J'} n} \mathrm{Mat}_{m_\lambda} \left(\sH_{\lambda_1}^{\tuple{\Lambda}^1}(q) \otimes \dots \otimes \sH_{\lambda_{p'}}^{\tuple{\Lambda}^{p'}}(q)\right).
\]
\end{theoreme}
In particular, the algebras $\sH_n^{\tuple{\Lambda}}(q, \tuple{v})$ and $\oplus_{\lambda \comp_{J'} n} \sH_{\lambda_1}^{\tuple{\Lambda}^1}(q) \otimes \dots \otimes \sH_{\lambda_{p'}}^{\tuple{\Lambda}^{p'}}(q)$ are Morita equivalent.
Note that since the following condition is satisfied (recall \eqref{equation:isom_BK-condition_v}):
\[
\prod_{1 \leq j < j' \leq p'} \prod_{i, i' \in I} \prod_{-n < a < n} \left(q^a (v_j q^i) - v_{j'} q^{i'}\right) \in F^\times,
\]
the Morita equivalence is known by \cite[Theorem 1.1]{DiMa}. Therefore, Theorem~\ref{theorem:unexpected_corollary-oplus} provides an explicit isomorphism for which the Morita equivalence of \cite{DiMa} follows.

\begin{remarque}
If $\tuple{\Lambda}^1 = \dots = \tuple{\Lambda}^{p'}$, by \cite[Corollary 3.2]{PA} or \cite{Ro} we know that the algebra of Theorem~\ref{theorem:unexpected_corollary-oplus} is a \emph{cyclotomic Yokonuma--Hecke algebra of type A}, as introduced in \cite{ChPA}.
\end{remarque}

\section{A presentation for \texorpdfstring{$\H_{p, n}^{\tuple{\Lambda}}(q)$}{HpnLambdaq}}
\label{section:presentation}

In this section, we prove our second main result, given in Corollary~\ref{corollary:presentation_H(G(rpn))}: we give a cyclotomic quiver Hecke-like presentation for $\H_{p, n}^{\tuple{\Lambda}}(q)$. The key is to make a careful choice for the family $\{Q_a(\tuple{k})\}_{a, \tuple{k}}$.

\subsection{A nice family}
\label{subsection:nice_family}

We consider the quiver $\Gamma$ with vertex set $V = \{v_j q^i\}_{i \in I, j \in J'} \simeq K = I \times J'$ of \textsection\ref{subsection:statement}, where $v_1, \dots, v_{p'} \in F^\times$ satisfy \eqref{equation:isom_BK-condition_v}.
We give here a particular choice for the family $\{Q_a(\tuple{k})\}_{a, \tuple{k}}$. We recall the definition of the family $\{P_a(\tuple{k})\}_{a, \tuple{k}}$ of \eqref{equation:candidate_homomorphisms-P_a}.

\begin{lemme}[\protect{\cite[(5.4)]{StWe}}]
\label{lemma:choice_Qa_SW}
The family $\{Q_a(\tuple{k})\}_{1 \leq a < n, \tuple{k} \in K^n}$ given by:
\[
Q_a(\tuple{k}) \coloneqq \begin{cases}
1 - q + q y_{a+1} - y_a & \text{if } k_a = k_{a+1},
\\
\frac{1 - P_a(\tuple{k})}{y_{a+1} - y_a} & \text{if } k_a \leftarrow k_{a+1} \text{ or } k_a \leftrightarrows k_{a+1},
\\
1 - P_a(\tuple{k}) & \text{otherwise,}
\end{cases}
\]
satisfies the property $(\mathrm{BK})$.
\end{lemme}

\begin{remarque}
\label{remark:leftarrow_or_leftright}
The condition ``$k_a \leftarrow k_{a+1}$ or $k_a \leftrightarrows k_{a+1}$'' is equivalent to ``$v_{k_a} q^{k_a} = qv_{k_{a+1}} q^{k_{a+1}}$''. With $\tuple{k} = (\tuple{i}, \tuple{j})$, that means ``$i_a = i_{a+1} + 1$ and $j_a = j_{a+1}$''.
\end{remarque}

The family given in \cite[(4.36)]{BrKl} would be the following one:
\begin{equation}
\label{equation:choice_Qa_BK}
Q_a^{\mathrm{BK}}(\tuple{k}) \coloneqq 
\begin{cases}
1 - q + qy_{a+1} - y_a & \text{if } k_a = k_{a+1},
\\
(y_a(\tuple{k}) - q y_{a+1}(\tuple{k})) / (y_a(\tuple{k}) - y_{a+1}(\tuple{k})) & \text{if } k_a \nrelbar k_{a+1},
\\
(y_a(\tuple{k}) - q y_{a+1}(\tuple{k})) / (y_a(\tuple{k}) - y_{a+1}(\tuple{k}))^2 & \text{if } k_a \to k_{a+1},
\\
v_{k_a} q^{k_a} & \text{if } k_a \leftarrow k_{a+1},
\\
v_{k_a}q^{k_a} / (y_a(\tuple{k}) - y_{a+1}(\tuple{k})) & \text{if } k_a \leftrightarrows k_{a+1}.
\end{cases}
\end{equation}
We will see in Remark~\ref{remark:choice_BK_not_adapted} why the choice of Lemma~\ref{lemma:choice_Qa_SW} is more adapted to our problem.
For the convenience of the reader, we will now give a proof of Lemma~\ref{lemma:choice_Qa_SW}.

\begin{proof}[Proof of Lemma~\ref{lemma:choice_Qa_SW}]
First, let us prove that $Q_a(\tuple{k})$ is well-defined. If $k_a \neq k_{a+1}$ we have $P_a(\tuple{k}) = \frac{(1 - q)y_{a+1}(\tuple{k})}{y_{a+1}(\tuple{k}) - y_a(\tuple{k})}$, thus:
\[
1 - P_a(\tuple{k}) = \frac{q y_{a+1}(\tuple{k}) - y_a(\tuple{k})}{y_{a+1}(\tuple{k}) - y_a(\tuple{k})}
=
\frac{q v_{k_{a+1}} q^{k_{a+1}}(1 - y_{a+1}) - v_{k_a} q^{k_a}(1 - y_a)}{y_{a+1}(\tuple{k}) - y_a(\tuple{k})}.
\]
In particular, if $k_a \leftarrow k_{a+1}$ or $k_a \leftrightarrows k_{a+1}$ we get (recall Remark~\ref{remark:leftarrow_or_leftright}):
\[
1 - P_a(\tuple{k}) = \frac{v_{k_a} q^{k_a}(y_a - y_{a+1})}{y_{a+1}(\tuple{k}) - y_a(\tuple{k})},
\]
thus:
\[
\frac{1 - P_a(\tuple{k})}{y_{a+1} - y_a} = \frac{v_{k_a} q^{k_a}}{y_a(\tuple{k}) - y_{a+1}(\tuple{k})},
\]
which is well-defined.

As suggested in \cite{StWe}, we now notice that, if $k_a \neq k_{a+1}$:
\begin{equation}
\label{equation:proof_choice_Qa_SW}
(1 - P_a(s_a \cdot \tuple{k}))^{s_a} = q + P_a(\tuple{k}).
\end{equation}
This is a straightforward consequence of the equality
$P_a(\tuple{k}) + P_a(s_a \cdot \tuple{k})^{s_a} = 1 - q$ (see \cite[(4.28)]{BrKl}). Let us now check that $(\mathrm{BK})$ is satisfied. First, the element $Q_a(\tuple{k})$ is of course invertible when $k_a = k_{a+1}$ (since $1 - q \neq 0$), and the invertibility in the remaining cases follows from the above calculations so \eqref{equation:Qa_invertible} holds. Moreover, equation~\eqref{equation:Qa_=} is true by definition.

We now check the different relations \eqref{equation:QaQasa} involving $Q_a(\tuple{k})Q_a(s_a \cdot \tuple{k})^{s_a}$. If $k_a \nrelbar k_{a+1}$ (in particular, $k_a \neq k_{a+1}$) then $Q_a(\tuple{k}) = 1 - P_a(\tuple{k})$ and we immediately deduce \eqref{equation:Qa_nrelbar} from \eqref{equation:proof_choice_Qa_SW}. If $k_a \to k_{a+1}$ then $Q_a(\tuple{k}) = 1 - P_a(\tuple{k})$ and $Q_a(s_a \cdot \tuple{k}) = \frac{1 - P_a(s_a \cdot \tuple{k})}{y_{a+1} - y_a}$. Thus:
\[
Q_a(\tuple{k}) Q_a(s_a \cdot \tuple{k})^{s_a}
=
(1 - P_a(\tuple{k})) \frac{q + P_a(\tuple{k})}{y_a - y_{a+1}},
\]
so \eqref{equation:Qa_->} holds. The proof of \eqref{equation:Qa_<-} is similar.
If now $k_a \leftrightarrows k_{a+1}$ then $Q_a(\tuple{k}) = \frac{1 - P_a(\tuple{k})}{y_{a+1} - y_a}$ and $Q_a(s_a \cdot \tuple{k}) = \frac{1 - P_a(s_a \cdot \tuple{k})}{y_{a+1} - y_a}$, thus:
\[
Q_a(\tuple{k}) Q_a(s_a \cdot \tuple{k})^{s_a}
=
\frac{1 - P_a(\tuple{k})}{y_{a+1} - y_a}  \cdot \frac{q + P_a(\tuple{k})}{y_a - y_{a+1}},
\]
so \eqref{equation:Qa_<->} holds.

Finally, to prove equation~\eqref{equation:Qa_odd} it suffices to see that $P_{a+1}(s_{a+1}s_a \cdot \tuple{k})^{s_a} = P_a(s_a s_{a+1} \cdot \tuple{k})^{s_{a+1}}$. This equality follows from \cite[(4.29)]{BrKl} and the braid relation $s_a s_{a+1} s_a = s_{a+1} s_a s_{a+1}$.
\end{proof}

\begin{remarque}
We deduce from the calculations made at the beginning of the proof of Lemma~\ref{lemma:choice_Qa_SW} that $Q_a(\tuple{k}) = Q_a^{\mathrm{BK}}(\tuple{k})$ if $k_a = k_{a+1}$, $k_a \nrelbar k_{a+1}$ or $k_a \leftrightarrows k_{a+1}$.
\end{remarque}

\subsection{Intertwining}
\label{subsection:intertwining}

In this subsection, we show how our previous works allow us to prove our main result (Corollary~\ref{corollary:presentation_H(G(rpn))}).  For $j \in J'$, let us set $v_j \coloneqq \zeta^j$: it follows from the definition of $p'$ that $v_1, \dots, v_{p'}$ satisfy the distinct orbit condition \eqref{equation:isom_BK-condition_v}. In particular, the vertex set of $\Gamma$ is $\V = \{\zeta^j q^i\}_{i \in I, j \in J'}$.
Let us consider a weight $\tuple{\Lambda} = (\Lambda_k)_{k \in K}$ of level $r$, such that:
\begin{equation}
\label{equation:intertwining-Lambdak}
\Lambda_{i, j} = \Lambda_{i, j'} \eqqcolon \Lambda_i, \qquad \text{for all } i \in I \text{ and } j, j' \in J'.
\end{equation}
We suppose that the associated tuple $\tuple{\Lambda} = (\Lambda_i)_{i \in I}$, of level $\omega d$, satisfies the condition of Proposition~\ref{proposition:removing_repetitions-_Lambda}, that is (recall the notation $\eta$ of \eqref{equation:introduction-zetap'_qeta}):
\begin{equation}
\label{equation:intertwining-Lambda_i+eta}
\Lambda_i = \Lambda_{i + \eta}, \qquad \text{for all }  i \in I,
\end{equation}
so that the algebras $\H_n^{\tuple{\Lambda}}(q, \zeta), \H_{p, n}^{\tuple{\Lambda}}(q)$ (recall Definition~\ref{definition:removing_repetitions-Hn_Hpn}) and the shift automorphism of $\H_n^{\tuple{\Lambda}}(q, \zeta)$ (recall Proposition~\ref{proposition:G(r1n)-def_sigma}) are well-defined.
We will use the above condition \eqref{equation:intertwining-Lambda_i+eta} and the results of \textsection\ref{subsection:fixed_subalgebra} to define a particular automorphism $\sigma$ of $\sqH_n^{\tuple{\Lambda}}(\Gamma)$. 

Let us define $\sigma : \V \to \V$ by:
\begin{equation}
\label{equation:intertwining-definition_sigma}
\sigma(v) \coloneqq \zeta v,
\end{equation}
for all  $v \in \V$.
Note that $\sigma$ is well-defined since $\V$ is also given by $\{\zeta^j q^i\}_{i \in I, j \in J}$. Moreover, the reader may have noticed the similarity with the map of Proposition~\ref{proposition:G(r1n)-def_sigma}.

\begin{lemme}
\label{lemma:sigma_v->zeta_v}
The map $\sigma : \V \to \V$ defined on the vertices of $\Gamma$ satisfies the assumptions of \textsection\ref{subsection:fixed_subalgebra}, that is:
\begin{itemize}
\item the map $\sigma : \V \to \V$ is a bijection;
\item if $(\vertex{v}, \vertex{v}')$ is an edge of $\Gamma$ then $(\sigma(\vertex{v}), \sigma(\vertex{v}'))$ is also an edge of $\Gamma$;
\item for any $p_1 \in \{1, \dots, p-1\}$ and any vertex $\vertex{v} \in \V$ we have $\sigma^{p_1}(\vertex{v}) \neq \vertex{v} = \sigma^p(\vertex{v})$.
\end{itemize}
\end{lemme}

\begin{proof}
Since $\zeta$ is a primitive $p$th root of unity, we deduce that the first and third points are satisfied. It remains to prove the second one. Let $(\vertex{v}, \vertex{v}')$ be an edge of $\Gamma$. By definition, we have $\vertex{v}' = q\vertex{v}$, thus $\zeta \vertex{v}' = \zeta (q\vertex{v}) = q (\zeta \vertex{v})$. Hence, we have  $\sigma(\vertex{v}') = q \sigma(\vertex{v})$: we have proved that $(\sigma(\vertex{v}), \sigma(\vertex{v}'))$ is an edge of $\Gamma$.
\end{proof}

The action of $\sigma$ on $\V$ is algebraically easy. Let us now describe how $\sigma$ acts ``graphically'' on the set $\V$ of the vertices of $\Gamma$, that is, on $K = I \times J'$. Let $i \in I, j \in J'$ and set $v \coloneqq \zeta^j q^i$. We have:
\begin{equation}
\label{equation:intertwining-description_sigma_translation}
\sigma(\zeta^j q^i) = \zeta^{j+1} q^i.
\end{equation}
Hence, if $j < p'$ then $\sigma$ just translates the vertex $\vertex{v}$ to the copy of $\Gamma_e$ directly on its right. If $j = p'$, we have $j + 1  = p'+1 \notin J'$ thus we write:
\begin{equation}
\label{equation:intertwining-description_sigma_rotate}
\sigma(\zeta^{p'} q^i) = \zeta \zeta^{p'} q^i = \zeta q^{i + \eta}.
\end{equation}
It means that $v$ is translated to the first copy of $\Gamma_e$ \emph{and} rotated by $\eta$. Note that depending on $\Gamma$, there may not be any translation or rotation.  With the examples of Figure~\ref{figure:Gamma}, that gives:
\begin{description}
\item[case $(e, p) = (2, 3)$] we have $p' = 3, \eta = 0$ and the map $\sigma$ is given by the product of 3-cycles $(0_1, 0_2, 0_3)(1_1, 1_2, 1_3)$;
\item[case $(e, p) = (2, 6)$] we have $p' = 3, \eta = 1$ and the map $\sigma$ is given by the 6-cycle $(0_1, 0_2, 0_3, 1_1, 1_2, 1_3)$;
\item[case $(e, p) = (\infty, 2)$] we have $p' = 2, \eta = 0$ and the map $\sigma$ is given by the product of transpositions $\prod_{i \in I} (i_1, i_2)$.
\end{description}
In particular, note that $\sigma$ has indeed order $p$.

\bigskip
By Theorem~\ref{theorem:definition_sigma_quiver} and Lemma~\ref{lemma:sigma_Ilambda}, the permutation $\sigma$ of the vertices of $\V$ induces an isomorphism $\sqH_\alpha^{\tuple{\Lambda}}(\Gamma) \to \sqH_{\sigma \cdot \alpha}^{\sigma \cdot \tuple{\Lambda}}(\Gamma)$ for any $\alpha \comp_K n$. Let us now check that the weight $\tuple{\Lambda}$ satisfies the $\sigma$-stability condition \eqref{equation:weight_adapted}.

\begin{proposition}
\label{proposition:intertwining-Lambda_sigma_stable}
For any $k \in K = I \times J'$ we have $\Lambda_k = \Lambda_{\sigma(k)}$.
\end{proposition}

\begin{proof}
We have seen above that for $(i, j) \in I \times J'$:
\begin{itemize}
\item if $j < p'$ then $\sigma(i, j) = (i, j + 1)$;
\item if $j = p'$ then $\sigma(i, j) = (i + \eta, 1)$.
\end{itemize}
Thus, we deduce the result from \eqref{equation:intertwining-Lambdak} and \eqref{equation:intertwining-Lambda_i+eta}.
\end{proof}

By Lemma~\ref{lemma:sigma_I[alpha]}, we know that the map $\sigma$ induces an automorphism of the cyclotomic quiver Hecke algebra $\qH_n^{\tuple{\Lambda}}(\Gamma)$. We will refer to it as the \emph{shift automorphism} of $\qH_n^{\tuple{\Lambda}}(\Gamma)$.

\begin{lemme}
\label{lemma:shift_ya_Pa_Qa}
The power series $y_a(\tuple{k}), P_a(\tuple{k})$ and $Q_a(\tuple{k})$ of $\qH_n^{\tuple{\Lambda}}(\Gamma)$ are shift-invariant. Moreover:
\begin{gather*}
y_a(\sigma(\tuple{k}))= \zeta y_a(\tuple{k}),
\\
P_a(\sigma(\tuple{k})) = P_a(\tuple{k}),
\\
 Q_a(\sigma(\tuple{k})) = Q_a(\tuple{k}).
\end{gather*}
\end{lemme}

\begin{proof}
The first statement is clear since $y_a$ and $y_{a+1}$ are shift-invariant (by definition, just recall Theorem~\ref{theorem:definition_sigma_quiver}). Recall that $\V = \{\zeta^j q^i\}_{i \in I, j \in J'} \simeq K$ is the vertex set of $\Gamma$. The image of $\tuple{k} \in K^n$ in $V^n$ is $(\zeta^{k_1}q^{k_1}, \dots, \zeta^{k_n} q^{k_n})$, where $\zeta^k q^k = \zeta^j q^i$ if $k = (i, j)$ (recall \eqref{equation:definition_zetak_qk}).  In particular, the image of $\sigma(\tuple{k})$ in $V^n$ is $(\zeta \zeta^{k_1} q^{k_1}, \dots, \zeta \zeta^{k_n} q^{k_n})$. Thus, we have:
\[
y_a(\sigma(\tuple{k})) = \zeta \zeta^{k_a} q^{k_a}(1 - y_a) = \zeta y_a(\tuple{k}).
\]
Hence, if $k_a = k_{a+1}$ then by Lemma~\ref{lemma:->_compatible} we have $P_a(\tuple{k}) = P_a(\sigma(\tuple{k}))$, and if $k_a \neq k_{a+1}$ we have:
\begin{align*}
P_a(\sigma(\tuple{k}))
&=
(1 - q){\left(1 - y_a(\sigma(\tuple{k})) y_{a+1}(\sigma(\tuple{k}))^{-1}\right)}^{-1}
\\
&=
(1 - q){\left(1 - \zeta \zeta^{-1} y_a(\tuple{k})y_{a+1}(\tuple{k})^{-1}\right)}^{-1}
\\
&=
P_a(\tuple{k}).
\end{align*}
The last equality $Q_a(\sigma(\tuple{k})) = Q_a(\tuple{k})$ is now obvious.
\end{proof}

Let us now denote by $\widetilde{\sigma} : \H_n^{\tuple{\Lambda}}(q, \zeta) \to \H_n^{\tuple{\Lambda}}(q, \zeta)$ the shift automorphism of $\H_n^{\tuple{\Lambda}}(q, \zeta)$ (defined in Proposition~\ref{proposition:G(r1n)-def_sigma}). Recalling the choice for $\tuple{v}$ that we made at the beginning of \textsection\ref{subsection:intertwining}, we consider the $F$-algebra isomorphism $f: \H_n^{\tuple{\Lambda}}(q, \zeta) \to \qH_n^{\tuple{\Lambda}}(\Gamma)$ from Theorems~\ref{theorem:BK_generalised} and \ref{theorem:BK_generalised_block}, defined with the family $\{Q_a(\tuple{k})\}_{a, \tuple{k}}$ of Lemma~\ref{lemma:choice_Qa_SW}. 

\begin{theoreme}[Main theorem]
\label{theorem:sigmaalpha_alphasigma}
We have $\sigma^{-1} \circ f = f \circ \widetilde{\sigma}$.
\end{theoreme}

\begin{proof}
Since we deal with algebra homomorphisms, it suffices to check the equality on the generators $S, T_1, \dots, T_{n-1}$ of $\H_n^{\tuple{\Lambda}}(q, \zeta)$. We successively have, using Lemma~\ref{lemma:shift_ya_Pa_Qa} (recall that, by definition, $S = X_1$):
\begin{align*}
\sigma^{-1} \circ f(S)
&=
\sum_{\tuple{k} \in K^n} \sigma^{-1}\big(y_1(\tuple{k}) e(\tuple{k})\big)
\\
&=
\sum_{\tuple{k} \in K^n} y_1(\tuple{k}) e(\sigma^{-1}(\tuple{k}))
\\
&=
\sum_{\tuple{k} \in K^n} y_1(\sigma(\tuple{k})) e(\tuple{k})
\\
&=
\zeta f(S)
\\
&=
f(\zeta S)
\\
&=
f \circ \widetilde\sigma (S),
\end{align*}
and:
\begin{align*}
\sigma^{-1} \circ f(T_a)
&=
\sum_{\tuple{k} \in K^n} \sigma^{-1}\big([\psi_a Q_a(\tuple{k}) - P_a(\tuple{k})] e(\tuple{k})\big)
\\
&=
\sum_{\tuple{k} \in K^n} [\psi_a Q_a(\tuple{k}) - P_a(\tuple{k})] e(\sigma^{-1}(\tuple{k}))
\\
&= \sum_{\tuple{k} \in K^n} [\psi_a Q_a(\sigma(\tuple{k})) - P_a(\sigma(\tuple{k}))] e(\tuple{k})
\\
&= \sum_{\tuple{k} \in K^n} [\psi_a Q_a(\tuple{k}) - P_a(\tuple{k})] e(\tuple{k})
\\
&=
f(T_a)
\\
&=
f \circ \widetilde\sigma(T_a).
\end{align*}
Note that the above sums over $K^n$ are in fact finite, since all but finitely many $e(\tuple{k}) \in \qH_n^{\tuple{\Lambda}}(\Gamma)$ are zero (recall, for instance, \textsection\ref{subsubsection:conclusion}).
\end{proof}

\begin{remarque}
\label{remark:choice_BK_not_adapted}
Theorem~\ref{theorem:sigmaalpha_alphasigma} fails if we consider the homomorphism $f$ built from the family $\{Q_a^\mathrm{BK}(\tuple{k})\}_{a, \tuple{k}}$. For instance, Lemma~\ref{lemma:shift_ya_Pa_Qa} is no longer valid with $Q_a^{\mathrm{BK}}(\tuple{k})$, since if $k_a \leftarrow k_{a+1}$:
\[
Q_a^{\mathrm{BK}}(\sigma(\tuple{k})) = \zeta \zeta^{k_a} q^{k_a} = \zeta Q_a^{\mathrm{BK}}(\tuple{k}),
\]
and the same result holds if $k_a \to k_{a+1}$.
\end{remarque}

We now recall some notation and facts from \textsection\ref{subsection:fixed_subalgebra}.
If $\alpha \comp_K n$, we denote by $[\alpha]$ its orbit under the action of $\langle \sigma\rangle$ (this action is defined in Lemma~\ref{lemma:fixed_subalgebra-sigma_cdot_alpha}). We have an associated subset $K^{[\alpha]} = \sqcup_{\beta \in [\alpha]} K^\beta$ of $K^n$ (see \eqref{equation:affine_case-K[alpha]}).
The quotient set of $K^{[\alpha]}$ by the equivalence relation~$\sim$ generated by $\tuple{k} \sim \sigma(\tuple{k})$ for all $\tuple{k} \in K^{[\alpha]}$ is  $K^{[\alpha]}_\sigma$ (Definition~\ref{definition:affine_case-Kalphasigma}).
Each equivalence class $\gamma \in K^{[\alpha]}_\sigma$ has cardinality $p$, and is given by $\gamma = \{\tuple{k}, \sigma(\tuple{k}), \dots, \sigma^{p-1}(\tuple{k})\}$ for some $\tuple{k} \in K^{[\alpha]}$ (see \eqref{equation:gamma_p'}).
Finally, thanks to the canonical map $K^n / {\sim} \to (K / {\sim})^n$ and Lemma~\ref{lemma:sigma_v->zeta_v}, for any $\gamma \in K_\sigma^{[\alpha]}$ and $a \in \{1, \dots, n\}$ we have well-defined statements  $\gamma_a = \gamma_{a+1}, \gamma_a \to \gamma_{a+1}$, etc. (see Lemma~\ref{lemma:->_compatible} and before Remark~\ref{remark:gammaa_gamma'a}). Moreover, since $\tuple{\Lambda}$ is $\sigma$-stable (Proposition~\ref{proposition:intertwining-Lambda_sigma_stable}) the integer $\Lambda_{\gamma_a}$ is well-defined.

\begin{corollaire}
\label{corollary:presentation_H(G(rpn))}
The $F$-algebra isomorphism $f : \H_n^{\tuple{\Lambda}}(q, \zeta) \to \qH_n^{\tuple{\Lambda}}(\Gamma)$ induces an isomorphism between $\H_{p, n}^{\tuple{\Lambda}}(q)$ and ${\qH_n^{\tuple{\Lambda}}(\Gamma)}^\sigma$.
Hence, we have the following $F$-algebra isomorphism:
\[
\H_{p, n}^{\tuple{\Lambda}}(q) \simeq \bigoplus_{[\alpha]} \qH_{[\alpha]}^{\tuple{\Lambda}}(\Gamma)^\sigma,
\]
where $[\alpha]$ runs over the orbits of the $K$-compositions of $n$ under the action of $\langle \sigma\rangle$, and the subalgebra $\H_{p, [\alpha]}^{\tuple{\Lambda}}(q)$ has a presentation given by the generators
\[
\{e(\gamma)\}_{\gamma \in K^{[\alpha]}_\sigma} \cup \{y_1, \dots, y_n\} \cup \{\psi_1, \dots, \psi_{n-1}\},
\]
and the relations
\eqref{relations:affine_quiver_fixed} and \eqref{relation:quiver_cyclo_fixed}.
\end{corollaire}

\begin{proof}
Using Theorem~\ref{theorem:sigmaalpha_alphasigma}, for $h \in \H_n^{\tuple{\Lambda}}(q, \zeta)$ we have:
\[
\widetilde{\sigma}(h) = h \iff f \circ \widetilde{\sigma}(h) = f(h) \iff \sigma^{-1} \circ f(h) = f(h) \iff f(h) = \sigma \circ f(h),
\]
hence:
\[
h \text{ is fixed under } \widetilde\sigma \iff f(h) \text{ is fixed under } \sigma.
\]
Using Corollary~\ref{corollary:H(G(rpn))_fixe}, we get:
\[
\H_{p, n}^{\tuple{\Lambda}}(q)
\simeq
\H_n^{\tuple{\Lambda}}(q, \zeta)^{\tilde{\sigma}}
\simeq
\qH_n^{\tuple{\Lambda}}(\Gamma)^\sigma,
\]
as desired. We deduce the second statement from the equality $\qH_n^{\tuple{\Lambda}}(\Gamma)^\sigma = \oplus_{[\alpha]} \qH_{[\alpha]}^{\tuple{\Lambda}}(\Gamma)^\sigma$ (note that this direct sum is finite by Theorem~\ref{theorem:BK_generalised_block})  and Theorem~\ref{theorem:quiver_two_subalgebra_same}, where we gave a presentation for $\qH_{[\alpha]}^{\tuple{\Lambda}}(\Gamma)^\sigma$.
\end{proof}


Recall from Remark~\ref{remark:grading_qH_cyclo_fixed} that $\qH_n^{\tuple{\Lambda}}(\Gamma)$ is naturally $\mathbb{Z}$-graded. From this grading, Theorem~\ref{theorem:BK_generalised} and  the isomorphism $f$, we can endow $\H_n^{\tuple{\Lambda}}(q, \zeta)$ with a (non-trivial) $\mathbb{Z}$-grading.

\begin{corollaire}
\label{corollary:intertwining-graded_subalgebra}
The shift automorphism $\widetilde\sigma : \H_n^{\tuple{\Lambda}}(q, \zeta) \to \H_n^{\tuple{\Lambda}}(q, \zeta)$ is homogeneous with respect to the previous grading. Moreover,  the subalgebra $\H_{p, n}^{\tuple{\Lambda}}(q)$ is a graded subalgebra of $\H_n^{\tuple{\Lambda}}(q, \zeta)$.
\end{corollaire}

\begin{proof}
Recall from Remark~\ref{remark:grading_qH_cyclo_fixed} that $\sigma : \qH_n^{\tuple{\Lambda}}(\Gamma) \to \qH_n^{\tuple{\Lambda}}(\Gamma)$ is homegeneous and that $\qH_n^{\tuple{\Lambda}}(\Gamma)^\sigma$ is a graded subalgebra. We thus deduce the first assertion from Theorem~\ref{theorem:sigmaalpha_alphasigma} and the second one from Corollary~\ref{corollary:presentation_H(G(rpn))}.
\end{proof}

We now give an analogue of a classical corollary of \cite[Theorem 1.1]{BrKl}.

\begin{corollaire}
\label{corollary:intertwining-independence_q}
If $\tilde{q} \in F \setminus \{0, 1\}$ has the same order $e \in \mathbb{N}_{\geq 2} \cup \{\infty\}$ as $q$ then:
\[
\H_{p, n}^{\tuple{\Lambda}}(\tilde{q}) \simeq \H_{p, n}^{\tuple{\Lambda}}(q),
\]
as $F$-algebras.
\end{corollaire}

\begin{proof}
We know from Lemma~\ref{lemma:introduction-p'_gcd} and Theorem~\ref{theorem:BK_generalised} that the algebras $\H_n^{\tuple{\Lambda}}(q)$ and $\H_n^{\tuple{\Lambda}}(\tilde{q})$ are isomorphic to the same cyclotomic quiver Hecke algebra $\qH_n^{\tuple{\Lambda}}(\Gamma)$. Moreover, we have the following isomorphism:
\[
\H_{p, n}^{\tuple{\Lambda}}(q) \simeq \qH_n^{\tuple{\Lambda}}(\Gamma)^\sigma,
\]
where $\sigma$ is uniquely determined by the quiver $\Gamma$ and the element $\eta \in I$ such that $q^\eta = \zeta^{p'}$. To prove that $\H_{p, n}^{\tuple{\Lambda}}(\tilde{q}) \simeq \H_{p, n}^{\tuple{\Lambda}}(q)$, it thus suffices to prove that
there is a primitive $p$th root of unity $\tilde{\zeta} \in F^\times$ such that:
\[
\tilde{q}^\eta = \tilde{\zeta}^{p'}
\]
(recall from Lemma~\ref{lemma:introduction-p'_gcd} that~$p'$ does not depend on the chosen primitive $p$th root of unity).
To deal with the case $e = \infty$, it suffices to set $\tilde{\zeta} \coloneqq \zeta$. Recall that, in that case, we have $\eta = 0$ and $p' = p$. Hence, we now assume that $e < \infty$. Since $q$ and $\tilde{q}$ are both primitive $e$th roots of unity, there is some $a \in \mathbb{Z}$, invertible modulo $e$, such that $\tilde{q} = q^a$. In particular, for any $k \in \mathbb{Z}$ we have $\tilde{q} = q^{a + ke}$. Since $q^\eta = \zeta^{p'}$, we get:
\[
\tilde{q}^\eta = {(\zeta^{a+ke})}^{p'}.
\]
Therefore, it suffices to prove that there is some $k \in \mathbb{Z}$ such that $a + ke$ is invertible modulo~$p$, that is, such that $\tilde{\zeta} \coloneqq \zeta^{a + ke}$ is a primitive $p$th root of unity. A quick (but very powerful) argument is to use Dirichlet's theorem about arithmetic progression (see also \cite[Lemma 3.5]{Hu_crystal}): since $a$ and $e$ are coprime, the set $\{a + ke\}_{k \in \mathbb{N}}$ contains infinitely many prime numbers. In particular, it contains a prime $\wp$ which does not divide $p$, hence which is coprime to $p$. It now suffices to choose $k \in \mathbb{N}$ such that $\wp = a + ke$.
\end{proof}

%
%

\appendix

\section{About \texorpdfstring{$\H_{p, n}^{\tuple{\Lambda}}(q)$}{HpnLambdaq}}
\label{section:appendix_induction}

The aim of this appendix is to give details for the statements of Remarks~\ref{remark:G(rpn)_case_p=1} and \ref{remark:hecke_G(rpn)-equivalent_presentation}.

\subsection{Case \texorpdfstring{$p = 1$}{p=1}}
\label{subsection:appendix-p1}

We prove here the statement of Remark~\ref{remark:G(rpn)_case_p=1}: the presentation of $\H_{p, n}^{\tuple{\varLambda}}(q)$ given at \textsection\ref{subsection:G(rpn)} gives the Ariki--Koike algebra $\H_n^{\tuple{\varLambda}}(q) = \H_n^{\tuple{\varLambda}}(q, 1)$ of \textsection\ref{subsection:Gr1n}. Note that since $p = 1$, the relation \eqref{relation:hecke_G(rpn)_bigbraid} becomes $s t'_1 = t_1 s$, thus we have:
\begin{equation}
\label{relation:appendix_tresse_p=1}
t'_1 = s^{-1} t_1 s.
\end{equation}

The result follows from the theorem below.

\begin{theoreme}
\label{theorem:appendix-case_p=1}
The algebra homomorphisms $\phi : \H_{1, n}^{\tuple{\varLambda}}(q) \to \H_n^{\tuple{\varLambda}}(q)$ and $\psi : \H_n^{\tuple{\varLambda}}(q) \to \H_{1, n}^{\tuple{\varLambda}}(q)$  given by:
\begin{align*}
\phi(s) &\coloneqq S,
\\
\phi(t'_1) &\coloneqq S^{-1} T_1 S,
\\
\phi(t_a) &\coloneqq T_a, \qquad \text{for all } a \in \{1, \dots, n-1\},
\\
\intertext{and:}
\psi(S) &\coloneqq s,
\\
\psi(T_a) &\coloneqq t_a, \qquad \text{for all }  a \in \{1, \dots, n-1\},
\end{align*}
are well-defined and inverse to each other.
\end{theoreme}

\begin{proof}
We first check that $\psi$ is an algebra homomorphism: all relations are straightforward except \eqref{relation:hecke_ST1ST1}, but it follows from \eqref{relation:hecke_G(rpn)_braid3} and \eqref{relation:appendix_tresse_p=1}. Concerning $\phi$, again all relations are straightforward, except \eqref{relation:hecke_G(rpn)_braid6} (if $n \geq 3$). Note the following consequence of \eqref{relation:hecke_ST1ST1}:
\begin{equation}
\label{relation:appendix-ST1ST1_modified}
S^{-1} T_1 S T_1 = T_1 S T_1 S^{-1}.
\end{equation}
In the following calculation, we adopt the following conventions:
\begin{itemize}
\item we use color when a quantity simplifies;
\item we use underbrace when we will use a relation;
\item we use parenthesis when we did use a relation.
\end{itemize}
We have:
\begin{align*}
&\phi(t'_1) \phi(t_1) \phi(t_2) \phi(t'_1) \phi(t_1) \phi(t_2)
=
\phi(t_2) \phi(t'_1) \phi(t_1) \phi(t_2) \phi(t'_1) \phi(t_1)
\\
\iff&
[S^{-1} T_1 S] T_1 T_2 [\underbrace{S^{-1} T_1 S] T_1}_{\eqref{relation:appendix-ST1ST1_modified}} T_2
=
\underbrace{T_2 [S^{-1}}_{\eqref{relation:hecke_STa}} T_1 S] T_1 T_2 [\underbrace{S^{-1} T_1 S] T_1}_{\eqref{relation:appendix-ST1ST1_modified}}.
\\
\iff& 
\textcolor{blue}{S^{-1}}T_1 S \underbrace{T_1 T_2 (T_1}_{\eqref{relation:hecke_braid3}} S T_1 \underbrace{S^{-1}) T_2}
=
(\textcolor{blue}{S^{-1}}T_2) T_1 S \underbrace{T_1 T_2 (T_1}_{\eqref{relation:hecke_braid3}} S T_1 S^{-1})
\\
\iff& 
T_1 \underbrace{S (T_2} T_1 \underbrace{T_2) S} T_1 (T_2 \textcolor{blue}{S^{-1}})
=
T_2 T_1 \underbrace{S (T_2} T_1 \underbrace{T_2) S} T_1 \textcolor{blue}{S^{-1}}
\\
\iff&
T_1 (T_2 S) T_1 (S \underbrace{T_2) T_1 T_2}
=
\underbrace{T_2 T_1 (T_2} S) T_1  (S T_2) T_1
\\
\iff&
\textcolor{blue}{T_1 T_2} S T_1 S (T_1 \textcolor{green}{T_2 T_1})
=
(\textcolor{blue}{T_1 T_2} T_1) S T_1 S \textcolor{green}{T_2 T_1}
\\
\iff&
S T_1 S T_1 = T_1 S T_1 S,
\end{align*}
which allows us to conclude. Finally, the composition $\phi \circ \psi$ is the identity on the set of generators $\{S, T_1, \dots, T_{n-1}\}$, and using \eqref{relation:appendix_tresse_p=1} we find that $\psi \circ \phi$ is the identity on the set of generators $\{s, t'_1, t_1, \dots, t_{n-1}\}$. Hence, the algebras homomorphisms $\phi$ and $\psi$ are inverse isomorphisms and this concludes the proof.
\end{proof}

\subsection{Two equivalent relations}
\label{subsection:appendix-two_equivalent_relations}

We suppose that $p \geq 2$.
We prove here the statement of Remark~\ref{remark:hecke_G(rpn)-equivalent_presentation}: in the algebra $\H_{p, n}^{\tuple{\varLambda}}(q)$, the relations \eqref{relation:hecke_G(rpn)_bigbraid} and \eqref{relation:hecke_G(rpn)_strange} are equivalent. We will even prove a slightly more general statement  (cf. Proposition~\ref{proposition:appendix-braid}). Let $A$ be a unitary ring and $q \in A^\times$ an invertible element. Let $s, t'_1, t_1$ some symbols which satisfy:
\begin{equation}
\label{relation:ordre_ta}
(t'_1 + 1)(t'_1 - q) = (t_1 + 1)(t_1 - q) = 0.
\end{equation}

\begin{lemme}
\label{lemma:appendix-equality}
We have:
\[
{(q^{-1} t'_1 t_1)}^{2-p} t_1 s t'_1 + (q-1) \sum_{k = 1}^{p-2} {(q^{-1} t'_1 t_1)}^{1-k} s t'_1
=
(\underbrace{t_1^{-1} {t'_1}^{-1} t_1^{-1} {t'_1}^{-1}\dots}_{p-2 \text{ factors}}) (\underbrace{\dots t_1 t'_1 t_1 t'_1}_{p - 2 \text{ factors}}) t_1 s t'_1.
\]
\end{lemme}

\begin{proof}
For $p = 2$ we get:
\[
t_1 s t'_1 = t_1 s t'_1,
\]
which is obviously true. If the equality is satisfied for $p - 1 \geq 2$, we get:

\begin{gather*}
{(q^{-1} t'_1 t_1)}^{2-p} t_1 s t'_1 + (q-1) \sum_{k = 1}^{p-2} {(q^{-1} t'_1 t_1)}^{1-k} s t'_1
\\
=
(q^{-1} t'_1 t_1)^{2-p} t_1 s t'_1
+ (q-1) \sum_{k = 2}^{p-2} (q^{-1}t'_1 t_1)^{1-k} st'_1
+ (q-1) s t'_1
\\
=
(q^{-1} t'_1 t_1)^{-1} (q^{-1} t'_1 t_1)^{3-p} t_1 s t'_1
+ (q-1)(q^{-1} t'_1 t_1)^{-1} \sum_{k = 1}^{p - 3} (q^{-1}t'_1 t_1)^{1-k} st'_1
+ (q-1) s t'_1
\\
\displaybreak[0]
= (q^{-1} t'_1 t_1)^{-1} \left[ (q^{-1} t'_1 t_1)^{2-(p-1)} t_1 s t'_1
+ (q-1) \sum_{k = 1}^{(p-1) - 2} (q^{-1}t'_1 t_1)^{1-k} st'_1\right]
+ (q-1) s t'_1
\\
= (q^{-1} t'_1 t_1)^{-1} (\underbrace{t_1^{-1} {t'_1}^{-1}\dots}_{p-3}) (\underbrace{\dots t_1 t'_1}_{p-3}) t_1 s t'_1+ (q-1) s t'_1
\\
= q  (\underbrace{t_1^{-1} {t'_1}^{-1}\dots}_{p-1}) (\underbrace{\dots t_1 t'_1}_{p-3}) t_1 s t'_1+ (q-1) s t'_1.
\end{gather*}
Let us now distinguish between two cases. If $p$ is even, using $q t_1^{-1} = t_1 - (q-1)$ we get:
\begin{gather*}
{(q^{-1} t'_1 t_1)}^{2-p} t_1 s t'_1 + (q-1) \sum_{k = 1}^{p-2} {(q^{-1} t'_1 t_1)}^{1-k} s t'_1
\\
=
\textcolor{blue}{q} (\underbrace{t_1^{-1} {t'_1}^{-1}  \dots {t'_1}^{-1} \textcolor{blue}{t_1^{-1}}}_{p-1})(\underbrace{t'_1 t_1 \dots t_1 t'_1}_{p-3}) t_1 s t'_1 + (q-1) s t'_1
\\
=
(\underbrace{t_1^{-1} {t'_1}^{-1}  \dots {t'_1}^{-1}}_{p-2})[t_1 - (q-1)](\underbrace{t'_1 t_1 \dots t_1 t'_1}_{p-3}) t_1 s t'_1 + (q-1) s t'_1
\\
=
(\underbrace{t_1^{-1} {t'_1}^{-1}  \dots {t'_1}^{-1}}_{p-2})(\underbrace{t_1 t'_1 t_1 \dots t_1 t'_1}_{p-2}) t_1 s t'_1 - (q-1)(\underbrace{t_1^{-1} {t'_1}^{-1}  \dots {t'_1}^{-1}}_{p-2})(\underbrace{t'_1 t_1 \dots t_1 t'_1}_{p-3}) t_1 s t'_1 + (q-1) s t'_1
\\
= (\underbrace{t_1^{-1} {t'_1}^{-1}  \dots {t'_1}^{-1}}_{p-2})(\underbrace{t_1 t'_1 t_1 \dots t_1 t'_1}_{p-2}) t_1 s t'_1 - (q-1)t_1^{-1} t_1 s t'_1 + (q-1) s t'_1
\\
=(\underbrace{t_1^{-1} {t'_1}^{-1}  \dots {t'_1}^{-1}}_{p-2})(\underbrace{t_1 t'_1 t_1 \dots t_1 t'_1}_{p-2}) t_1 s t'_1,
\end{gather*}
thus we are done. If $p$ is odd, similarly we obtain, now using $q{t'_1}^{-1} = t'_1 - (q-1)$:
\begin{gather*}
{(q^{-1} t'_1 t_1)}^{2-p} t_1 s t'_1 + (q-1) \sum_{k = 1}^{p-2} {(q^{-1} t'_1 t_1)}^{1-k} s t'_1
\\
=
\textcolor{blue}{q} (\underbrace{t_1^{-1} {t'_1}^{-1}  \dots t_1^{-1} \textcolor{blue}{{t'_1}^{-1}}}_{p-1})(\underbrace{t_1 t'_1 \dots t_1 t'_1}_{p-3}) t_1 s t'_1 + (q-1) s t'_1
\\
=
(\underbrace{t_1^{-1} {t'_1}^{-1}  \dots t_1^{-1}}_{p-2})[t'_1 - (q-1)](\underbrace{t_1 t'_1 \dots t_1 t'_1}_{p-3}) t_1 s t'_1 + (q-1) s t'_1
\\
=
(\underbrace{t_1^{-1} {t'_1}^{-1}  \dots t_1^{-1}}_{p-2})(\underbrace{t'_1 t_1 t'_1 \dots t_1 t'_1}_{p-2}) t_1 s t'_1 - (q-1)(\underbrace{t_1^{-1} {t'_1}^{-1}  \dots t_1^{-1}}_{p-2})(\underbrace{t_1 t'_1 \dots t_1 t'_1}_{p-3}) t_1 s t'_1 + (q-1) s t'_1
\\
\displaybreak[0]
= (\underbrace{t_1^{-1} {t'_1}^{-1}  \dots t_1^{-1}}_{p-2})(\underbrace{t'_1 t_1 t'_1 \dots t_1 t'_1}_{p-2}) t_1 s t'_1 - (q-1)t_1^{-1} t_1 s t'_1 + (q-1) s t'_1
\\
=
(\underbrace{t_1^{-1} {t'_1}^{-1}  \dots t_1^{-1}}_{p-2})(\underbrace{t'_1 t_1 t'_1 \dots t_1 t'_1}_{p-2}) t_1 s t'_1,
\end{gather*}
thus we are done.
\end{proof}

\begin{proposition}
\label{proposition:appendix-braid}
We assume that $s, t'_1, t_1$ satisfy, in addition to \eqref{relation:ordre_ta}, the following relation:
\begin{equation}
\label{relation:mini_tresse}
s t'_1 t_1 = t'_1 t_1 s.
\end{equation}
The relations:
\begin{equation}
\tag{Ar}
\label{relation:ariki}
s t'_1 t_1 = {(q^{-1} t'_1 t_1)}^{2-p} t_1 s t'_1 + (q-1) \sum_{k = 1}^{p-2} {(q^{-1} t'_1 t_1)}^{1-k} s t'_1,
\end{equation}
and:
\begin{equation}
\tag{BMR}
\label{relation:bmr}
\underbrace{s t'_1 t_1 t'_1 t_1\dots}_{p + 1} = \underbrace{t_1 s t'_1 t_1 t'_1\dots}_{p+1},
\end{equation}
are equivalent.
\end{proposition}

\begin{proof}
By Lemma~\ref{lemma:appendix-equality}, relation \eqref{relation:ariki} is equivalent to:
\begin{equation}
\label{relation:proof}
s t'_1 t_1 =  (\underbrace{t_1^{-1} {t'_1}^{-1} t_1^{-1} {t'_1}^{-1}\dots}_{p-2 \text{ factors}}) (\underbrace{\dots t_1 t'_1 t_1 t'_1}_{p - 2 \text{ factors}}) t_1 s t'_1.
\end{equation}
If $p$ is even, this reads:
\[
st'_1 t_1 = (\underbrace{t_1^{-1} {t'_1}^{-1} \dots t_1^{-1} {t'_1}^{-1}}_{p-2})
(\underbrace{t_1 t'_1 \dots t_1 t'_1}_{p-2}) t_1 s t'_1,
\]
whence we obtain:
\[
\underbrace{t'_1 t_1 \dots t'_1 t_1}_{p-2} s t'_1 t_1 = \underbrace{t_1 t'_1 \dots t_1 t'_1}_{p-2} t_1 s t'_1.
\]
Thus, using \eqref{relation:mini_tresse} to bring $s$ to the left on both sides, we get:
\[
\underbrace{s t'_1 t_1 \dots t'_1 t_1}_{p + 1} = \underbrace{t_1 s t'_1 t_1 \dots t_1 t'_1}_{p+1},
\]
which is the desired result: the relations \eqref{relation:ariki} and \eqref{relation:bmr} are equivalent. If now $p$ is odd, relation \eqref{relation:proof} reads:
\[
st'_1 t_1 = (\underbrace{t_1^{-1} {t'_1}^{-1} \dots {t'_1}^{-1} t_1^{-1}}_{p-2})
(\underbrace{t'_1 t_1 \dots t_1 t'_1}_{p-2}) t_1 s t'_1,
\]
whence we obtain:
\[
\underbrace{t_1 t'_1 \dots t'_1 t_1}_{p-2} s t'_1 t_1 = \underbrace{t'_1 t_1 \dots t_1 t'_1}_{p-2} t_1 s t'_1.
\]
Thus, using \eqref{relation:mini_tresse} to bring $s$ to the left on both sides, we get:
\[
\underbrace{t_1 s t'_1 \dots t'_1 t_1}_{p + 1} = \underbrace{s t'_1 t_1 \dots t_1 t'_1}_{p+1},
\]
which is the desired result: the relations \eqref{relation:ariki} and \eqref{relation:bmr} are thus equivalent.
\end{proof}

Using Proposition~\ref{proposition:appendix-braid}, it is now clear that the algebra $\H_{p, n}^{\tuple{\varLambda}}(q)$ is isomorphic to the one defined by Ariki in \cite{Ar_rep} as stated in Remark~\ref{remark:hecke_G(rpn)-equivalent_presentation}.

\end{document}